\documentclass[11pt,reqno]{amsart}
\usepackage[T1]{fontenc}
\usepackage[utf8]{inputenc}
\usepackage{amsmath,amssymb,amsthm,mathrsfs}
\usepackage{xcolor}
\usepackage{tikz}
\usepackage{pgfplots}
\pgfplotsset{compat=1.16}
\usetikzlibrary{arrows.meta,positioning,patterns,decorations.pathreplacing,calc}
\definecolor{Shade}{gray}{0.90}
\pgfplotsset{
  every axis/.append style={
    axis lines=left,
    axis line style={gray!55,line width=.4pt},
    tick style={gray!55,line width=.4pt},
    tick align=outside,
    tick label style={font=\footnotesize},
    label style={font=\small},
    title style={font=\small,yshift=3pt},
    legend style={font=\footnotesize,draw=gray!45,fill=white,
      fill opacity=.92,text opacity=1,rounded corners=1pt,inner sep=3pt},
    every axis plot/.append style={line width=.9pt},
  },
}
\newlength{\figgap}
\tikzset{
  fig label/.style={font=\footnotesize},
  leader/.style={gray!70,line width=.3pt},
}
\usepackage[colorlinks=true,linkcolor=BrickRed,citecolor=RoyalBlue]{hyperref}
\usepackage{geometry}
\definecolor{BrickRed}{rgb}{0.72,0.16,0.11}
\definecolor{RoyalBlue}{rgb}{0.13,0.29,0.65}
\definecolor{ForestGreen}{rgb}{0.13,0.55,0.13}
\usepackage{cite}
\usepackage{palatino}

\newtheorem{theorem}{Theorem}[section]
\newtheorem{proposition}[theorem]{Proposition}
\newtheorem{lemma}[theorem]{Lemma}
\newtheorem{corollary}[theorem]{Corollary}
\theoremstyle{definition}
\newtheorem{definition}[theorem]{Definition}
\newtheorem{remark}[theorem]{Remark}
\newtheorem{example}[theorem]{Example}
\newtheorem{openproblem}[theorem]{Open Problem}
\newtheorem{assumption}[theorem]{Assumption}

\newcommand{\R}{\mathbb{R}}
\newcommand{\N}{\mathbb{N}}
\newcommand{\Om}{\Omega}
\newcommand{\dive}{\operatorname{div}}
\newcommand{\Lip}{\operatorname{Lip}}
\newcommand{\dist}{\operatorname{dist}}

\newcommand{\aaa}{\mathbf{a}}
\newcommand{\sg}{\sigma}
\newcommand{\Lam}{\Lambda}
\newcommand{\eps}{\varepsilon}
\newcommand{\nb}{\nabla}

\definecolor{NewEdit}{rgb}{0.62,0.00,0.55}

\numberwithin{equation}{section}
\begin{document}

\title[Born--Infeld model via superposition of infinitely many $p$-Laplacians]{Gradient
constraints, Born--Infeld,  and maximal surfaces via superposition of infinitely
many $p$-Laplacians}

\author{H. Abdel Hamid}
\address{Faculty of Sciences and Humanities, Fahad Bin Sultan University,
P.O.\ Box 15700, Tabuk 71454, Saudi Arabia}
\email{habdelhamid@fbsu.edu.sa}

\author{J. C. Chata Ortiz}
\address{Departamento de Matem\'atica, Universidade Federal de S\~ao Carlos --
UFSCar, 13565-905, S\~ao Carlos -- SP, Brazil}
\email{juan.carlos@unesp.br}

\author{F. Petitta}
\address{Dipartimento di Scienze di Base e Applicate per l'Ingegneria,
Sapienza Universit\`a di Roma, Via Scarpa 16, 00161 Roma, Italia}
\email{francesco.petitta@uniroma1.it}

\author{J. D. Rossi}
\address{Departamento de Matem\'aticas y Estad\'istica, Universidad Torcuato di Tella,
Campus Alcorta -- Av.\ Figueroa Alcorta 7350, Buenos Aires, Argentina}
\email{julio.rossi@utdt.edu}


\begin{abstract}
We study the Dirichlet problem for an infinite series of $p$-Laplacians,
\[
-\sum_{p=2}^{\infty}a_{p}\Delta_{p}u=f \ \text{ in }\Om,\qquad u=g\ \text{ on }\partial\Om,
\]
where $\{a_{p}\}$ is a sequence of nonnegative numbers whose associated power
series has radius of convergence $\sg\in(0,\infty]$. Such an operator is
formally $-\dive\big(\mathcal A(|\nb u|)\nb u\big)$ with $\mathcal A$ singular
at $|\nb u|=\sg$, and the model cases are the mean curvature operator in
Minkowski space and the Born--Infeld operator of nonlinear electrostatics.
The radius of convergence of the series forces the \emph{gradient constraint}
$\|\nb u\|_{\infty}\le\sg$ for the solutions, so that the natural variational problem is a
constrained one. We show that a unique minimizer of the associated energy
always exists (for boundary data compatible with the constraint) and always solves a variational inequality, and we identify the
\emph{saturation flux} $\Lam:=\sum_{p\ge2}a_{p}\sg^{p-1}=\sg^{-1}\sum_{p\ge2}a_p\sg^p$
as the quantity governing the solvability of the equation itself: if
$\Lam<\infty$ then a weak solution exists only if
$|\int_{E}f|\le\Lam P(E)$ for every set of finite perimeter $E\Subset\Om$, so
that for $f\equiv\lambda$ no solution exists as soon as
$\lambda>\Lam\, h(\Om)$, $h(\Om)$ being the Cheeger constant of $\Om$. The
threshold is sharp on balls, where the problem is solved explicitly and the
minimizer is shown to develop a \emph{gradient saturation region}
$\{|\nb u|=\sg\}$ of positive measure. On the other hand, when $\Lam=\infty$, which is the case for all Born--Infeld type operators,  no such
obstruction is present, and we prove that the minimizer solves the equation
whenever a Lipschitz bound below $\sg$ is available; {for $f\equiv0$ we
obtain such a bound, uniformly with respect to the truncations of the series,
under a bounded slope condition on the boundary datum}. As applications we obtain an isoperimetric bound on the charge
densities supported by a nonlinear electrostatics with saturating displacement,
and a quantitative convergence rate for the weak field expansion of the
Born--Infeld model.
\end{abstract}

\maketitle

\tableofcontents

\section{Introduction}\label{sec:intro}

Let $\Om$ be a bounded, connected, open subset of $\R^{N}$, $N\ge1$, with
Lipschitz boundary, let $\{a_{p}\}_{p\ge2}$ be a sequence of nonnegative real
numbers,
and consider the Dirichlet problem
\begin{equation}\label{eq:main}
\begin{cases}
\displaystyle -\sum_{p=2}^{\infty}a_{p}\Delta_{p}u=f &\text{in }\Om,\\[2mm]
u=g &\text{on }\partial\Om,
\end{cases}
\end{equation}
where $\Delta_{p}u=\dive(|\nb u|^{p-2}\nb u)$, $f$ is a given datum in $L^{1}(\Omega)$, and the boundary datum verifies $g\in C^{0,1}(\partial\Om)$.

Problem \eqref{eq:main} is the exact counterpart, for singularities in the
gradient variable, of a class of problems that has been studied for
singularities in the $u$ variable. Indeed, elliptic equations whose model is
\begin{equation}\label{eq:orsina}
-\dive\Big(\frac{\nb v}{1-v}\Big)=f\quad\text{in }\Om,\qquad v=0
\quad\text{on }\partial\Om,
\end{equation}
were studied in \cite{O2003} (see also \cite{ACLMOP2009,BLOP2011} for {formally} related problems). Expanding
the coefficient in power series, \eqref{eq:orsina} becomes formally
$-\sum_{m\ge1}\frac1m\Delta v^{m}=f$, a \emph{generalized porous medium
equation}, and the interplay between the radius of convergence of the series
and the solvability of the Dirichlet problem was analyzed in \cite{P2010} for $f\in L^{q}(\Omega)$, $q>\frac{N}{2}$: a
solution exists for every datum if and only if $\sum_{m}a_{m}\sg^{m}=+\infty$,
$\sg$ being the radius of convergence, while for $\sum_{m}a_{m}\sg^{m}<+\infty$
solutions cease to exist for large data, and the approximating sequences
develop \emph{flat zones} $\{v=\sg\}$ of positive measure.

Here we replace the singularity in $u$ by a singularity in $|\nb u|$. The
model case of \eqref{eq:main} is
\begin{equation}\label{eq:model}
-\dive\Big(\frac{\nb u}{(1-|\nb u|^{q})^{\alpha}}\Big)=f\quad\text{in }\Om,
\qquad u=g\quad\text{on }\partial\Om,
\end{equation}
with $q\in\N$, $\alpha>0$, which for $q=2$, $\alpha=\tfrac12$ is the
\emph{mean curvature operator in Minkowski space},
\begin{equation}\label{eq:mink}
-\dive\Big(\frac{\nb u}{\sqrt{1-|\nb u|^{2}}}\Big)=f ,
\end{equation}
governing spacelike hypersurfaces of prescribed mean curvature in the
Lorentz--Minkowski spacetime
\cite{BS1982,G1983,BCP2021,BJM2009,BJM2009b,BJM2014,COOR2013A,COOR2013B,BDD,Bayard},
and which is at the same time the Euler--Lagrange equation of the Born--Infeld
Lagrangian of nonlinear electrostatics \cite{BI1934,BdAP}. Expanding the
coefficient by the binomial series,
\[
\frac{1}{\sqrt{1-z^{2}}}=\sum_{k\ge0}\frac{1}{4^{k}}\binom{2k}{k}z^{2k},
\]
equation \eqref{eq:mink} takes precisely the form \eqref{eq:main} with
$a_{2k+2}=4^{-k}\binom{2k}{k}$ and $a_{p}=0$ for $p$ odd.

Two features distinguish \eqref{eq:main} from its zeroth order ancestor, and
they are the reason why the analysis is genuinely different. The first is that
the radius of convergence $\sg$ of the series now produces a constraint on the
\emph{gradient}: any function for which the left hand side of \eqref{eq:main}
makes sense satisfies $\|\nb u\|_{L^{\infty}(\Om)}\le\sg$ (Proposition
\ref{prop:constraint}). Problem \eqref{eq:main} is therefore intrinsically a
\emph{gradient constrained} problem, of the kind met in elastic--plastic
torsion and in sandpile growth models, and its natural variational
formulation is a constrained one. The second feature is    that the two natural
quantities attached to the sequence $\{a_{p}\}$, namely the total energy and
the total flux at the constraint,
\[
\Phi(\sg)=\sum_{p\ge2}\frac{a_{p}}{p}\sg^{p},
\qquad
\Lam=\sum_{p\ge2}a_{p}\sg^{p-1}=\frac1\sg\sum_{p\ge2}a_{p}\sg^{p},
\]
no longer coincide, and they play different roles: \emph{the existence of a
minimizer is governed by} $\Phi(\sg)$, \emph{whereas the solvability of the
equation is governed by} $\Lam$. Since $\Phi(\sg)\le\tfrac{\sg}{2}\Lam$, three
regimes are possible,
\begin{equation}\label{eq:trichotomy}
\text{(H1) } \Phi(\sg)=\infty;\qquad
\text{(H2) } \Phi(\sg)<\infty,\ \Lam=\infty;\qquad
\text{(H3) } \Lam<\infty ,
\end{equation}
and all three occur. In particular the Minkowski operator \eqref{eq:mink}
belongs to (H2), because $\Phi(1)=1$ while
$\Lam=\sum_{k}4^{-k}\binom{2k}{k}=+\infty$; this already shows that the naive
transposition of the criterion of \cite{P2010} for existence of solutions, which would be phrased in
terms of $\Phi(\sg)$, cannot be the right one, since the Dirichlet problem
for \eqref{eq:mink} is solvable for every bounded datum by the results of
Bartnik and Simon \cite{BS1982}; see the very recent work \cite{DFM2026} for
the case of unbounded prescribed curvature.

Our main results can be summarized as follows.

\smallskip

\noindent\emph{(i) The variational problem is always well posed.} If
$\|\nb\bar g\|_{\infty}\le\sg$ for some Lipschitz extension $\bar g$ of $g$ with
finite energy, then the functional
\[
F(v)=\int_{\Om}\Phi(|\nb v|)-\int_{\Om}fv
\]
has a unique minimizer $u$ over
$S_{g}=\{v\in C^{0,1}(\overline\Om):\|\nb v\|_{\infty}\le\sg,\ v=g
\text{ on }\partial\Om\}$, for every $f\in L^{1}(\Om)$ (Theorem
\ref{thm:min}). Moreover $u$ \emph{always} solves the variational inequality
\begin{equation}\label{eq:VIintro}
\sum_{p\ge2}a_{p}\int_{\Om}|\nb u|^{p-2}\nb u\cdot\nb(v-u)\ \ge\
\int_{\Om}f(v-u)\qquad\text{for every }v\in S_{g},\ F(v)<\infty,
\end{equation}
(Theorem \ref{thm:VI}), and it solves the equation as soon as the constraint
is inactive, i.e.\ as soon as $\|\nb u\|_{\infty}<\sg$ (Corollary
\ref{cor:VI2PDE}).

\smallskip
\noindent\emph{(ii) If $\Lam<\infty$, large data are forbidden.} If $u$ is a
weak solution of \eqref{eq:main} then the flux field
$\aaa(\nb u)=\sum_{p}a_{p}|\nb u|^{p-2}\nb u$ is bounded by $\Lam$, whence
\begin{equation}\label{eq:giustiintro}
\Big|\int_{E}f\Big|\ \le\ \Lam\,P(E)\qquad\text{for every set of finite
perimeter } E\Subset\Om ,
\end{equation}
(Theorem \ref{thm:nonexistence}). For $f\equiv\lambda$ this means that no weak
solution exists as soon as $\lambda>\Lam\,h(\Om)$, where $h(\Om)$ is the
Cheeger constant of $\Om$. Condition \eqref{eq:giustiintro} is the exact
analogue, in the present setting, of the classical necessary condition of
Giusti and Massari--Miranda for the prescribed mean curvature equation in
Euclidean space; in the language of \cite{P2010}, the first eigenvalue of the
$2$-Laplacian is here replaced by the first eigenvalue of the $1$-Laplacian.
We show in Section \ref{sec:radial} that the threshold $\Lam h(\Om)$ is attained on
balls, and that above it the minimizer develops a \emph{gradient saturation
region} $\{|\nb u|=\sg\}$ of positive measure, which is the exact counterpart
of the flat zones of \cite{P2010}, of the light rays of \cite{BS1982} and of
the plastic regions of elastic--plastic torsion.

\smallskip

\noindent\emph{(iii) If $\Lam=\infty$, everything reduces to a Lipschitz
bound.} The obstruction \eqref{eq:giustiintro} disappears, and by Corollary
\ref{cor:VI2PDE} the only missing ingredient for solvability is an a priori
estimate $\|\nb u\|_{\infty}\le\gamma<\sg$. We prove that for
$f$ constant the global Lipschitz constant of the minimizer equals its
boundary--to--interior slope (Theorem \ref{thm:haarrado}), so that the whole
question reduces to a \emph{boundary} gradient estimate; and we prove that
 {for $f\equiv0$ and} under the bounded slope condition of rank $K<\sg$ the
bound holds with $\gamma=K$, \emph{uniformly with respect to the truncations of
the series} (Theorem \ref{thm:bsc}). {Whenever such a uniform bound is
available}, the natural approximation scheme obtained by truncating the series
converges geometrically, with rate $K/{\sigma}$ (Theorem
\ref{thm:trunc}).

\smallskip
\noindent\emph{(iv) Consequences for the applied models.} Reading
\eqref{eq:giustiintro} in the language of nonlinear electrostatics, where
$-\nb u=\mathbf E$ is the electric field and $\aaa(\nb u)=\mathbf D$ the
electric displacement, the inequality is nothing but Gauss' law combined with
the bound $|\mathbf D|\le\Lam$: a charge density $\rho=f$ is admissible 
only if it satisfies an isoperimetric bound. Born--Infeld electrostatics is
precisely the borderline theory in which $|\mathbf E|$ is bounded but
$|\mathbf D|$ is not, which is why  no obstruction of this kind arises;
constitutive laws with saturating displacement do not, and our theorem
gives an explicit upper bound on the charge they can support. On the geometric side,
the same statement contains both the Euclidean obstruction and the Minkowskian
non-obstruction for prescribed mean curvature as the two cases
$\Lam<\infty$ and $\Lam=\infty$ of one single criterion. Finally, the
truncated problems are standard $(2,n)$-growth problems, so that Theorem
\ref{thm:trunc} provides a rigorous convergence statement, with a rate, for
the weak field expansion of the Born--Infeld model.

\smallskip

The paper is organized as follows: Section \ref{sec:setting} contains the
assumptions, the basic properties of the three functions related to
$\{a_{p}\}$, and a characterization of the operators representable as series of
$p$-Laplacians. Section \ref{sec:var} deals with the variational problem, the
variational inequality and the passage to the equation. Section
\ref{sec:existence} contains the Lipschitz estimates and the existence results
in the regime $\Lam=\infty$; Section \ref{sec:nonexistence}, the nonexistence
result in the regime $\Lam<\infty$. Section \ref{sec:radial} solves the radial
problem explicitly and proves the sharpness of the threshold, and Section
\ref{sec:1d} treats the one dimensional case. Applications and open problems
are discussed in Sections \ref{sec:applications} and \ref{sec:open}.
 {The  radial and one dimensional reductions of Section \ref{sec:radial} and
Section \ref{sec:1d} also make all the phenomena described above accessible to
numerics; this will be the content of a subsequent study.}

\medskip

\noindent\textbf{Notation.} $C$ denotes a positive constant which may change
from line to line and never depends on the solution. $P(E)$ is the perimeter of
a set of finite perimeter $E$ in $\R^{N}$, $|E|$ its Lebesgue measure, and
\[
h(\Om)=\inf\Big\{\frac{P(E)}{|E|}\ :\ E\Subset\Om,\ 0<|E|<\infty\Big\}
\]
the Cheeger constant of $\Om$. We write $T_{k}(s)=\max(-k,\min(k,s))$.

\section{Setting, assumptions, and the class of admissible operators}
\label{sec:setting}

\subsection{The three functions related to the sequence}

Throughout the paper we assume
\begin{equation}\label{eq:H}
a_{p}\ge0\ \ (p\ge2),\qquad a_{2}>0,\qquad a_{p}\neq0
\ \text{ for infinitely many } p .
\end{equation}
The assumption $a_{2}>0$ is the counterpart of the assumption $a_{1}\neq0$ of
\cite{P2010}: it guarantees nondegeneracy of the operator at $\nb u=0$; the
last assumption in \eqref{eq:H} is what makes the problem genuinely different
from a finite sum of $p$-Laplacians.

Related to $\{a_{p}\}$ there are three relevant functions,
\begin{equation}\label{eq:AphiPhi}
\mathcal A(t)=\sum_{p\ge2}a_{p}t^{p-2},
\qquad
\varphi(t)=t\,\mathcal A(t)=\sum_{p\ge2}a_{p}t^{p-1},
\qquad
\Phi(t)=\int_{0}^{t}\varphi(s)\,ds=\sum_{p\ge2}\frac{a_{p}}{p}t^{p},
\end{equation}
which we call respectively the \emph{coefficient}, the \emph{flux} and the
\emph{energy density} of the problem, and the vector field
\begin{equation}\label{eq:field}
\aaa(\xi)=\mathcal A(|\xi|)\,\xi=\sum_{p\ge2}a_{p}|\xi|^{p-2}\xi=\nb_{\xi}\big[\Phi(|\xi|)\big],
\qquad \xi\in\R^{N},\ |\xi|<\sg .
\end{equation}
With this notation problem \eqref{eq:main} reads
$-\dive\big(\aaa(\nb u)\big)=f$, i.e.\
$-\dive\big(\mathcal A(|\nb u|)\nb u\big)=f$.

\begin{remark}\label{rem:notation}
It is worth stressing the shift of index in \eqref{eq:AphiPhi}: writing the
operator as $\dive(\mathcal A(|\nb u|)\nb u)$ forces $\mathcal A$ to be the
power series $\sum_{p}a_{p}t^{p-2}$, so that $\mathcal A(0)=a_{2}>0$. For
instance, for the Minkowski operator \eqref{eq:mink} one has
$\mathcal A(t)=(1-t^{2})^{-1/2}$ and $a_{2}=1$.
\end{remark}

Let
\begin{equation}\label{eq:sigma}
\frac{1}{\sg}=\limsup_{p\to\infty}a_{p}^{1/p}\ \in[0,+\infty],
\end{equation}
so that $\sg\in[0,+\infty]$ is the common radius of convergence of the three
series in \eqref{eq:AphiPhi}. We shall mostly assume $0<\sg<\infty$; the two extreme
cases, $\sigma=0$ and $\sigma = \infty$, are discussed in Remark \ref{rem:extreme}. We finally set
\begin{equation}\label{eq:LamPhi}
\Lam:=\varphi(\sg^{-})=\sum_{p\ge2}a_{p}\sg^{p-1}\in(0,+\infty],
\qquad
\Phi(\sg):=\Phi(\sg^{-})=\sum_{p\ge2}\frac{a_{p}}{p}\sg^{p}\in(0,+\infty].
\end{equation}

\begin{lemma}\label{lem:elementary}
Assume \eqref{eq:H} and $0<\sg<\infty$. Then
\begin{enumerate}
\item[(i)] $\mathcal A$ is nondecreasing on $[0,\sg)$ with
$\mathcal A\ge a_{2}>0$; $\varphi$ is an increasing homeomorphism of
$[0,\sg)$ onto $[0,\Lam)$ with $\varphi'\ge a_{2}$;
\item[(ii)] $\Phi$ is strictly convex and increasing on $[0,\sg)$ and
$\Phi(t)\ge\frac{a_{2}}{2}t^{2}$;
\item[(iii)] $\xi\mapsto\Phi(|\xi|)$ is strictly convex on
$\{|\xi|<\sg\}$ and $\aaa=\nb_{\xi}\Phi(|\cdot|)$ is strictly monotone there;
\item[(iv)] $\displaystyle \Phi(\sg)\le\frac{\sg\Lam}{2}$; in particular
$\Lam<\infty\Rightarrow\Phi(\sg)<\infty$, and the trichotomy
\eqref{eq:trichotomy} is exhaustive and mutually exclusive.
\end{enumerate}
\end{lemma}

\begin{proof}
(i)--(iii) are immediate from \eqref{eq:AphiPhi} and $a_{p}\ge0$, $a_{2}>0$;
note $\varphi'(t)=\sum_{p}(p-1)a_{p}t^{p-2}\ge a_{2}$, and that
$\Phi(|\xi|)$ is convex as a nondecreasing convex function of the norm, strictly
so because $\Phi''\ge a_{2}$. For (iv),
$\Phi(\sg)=\sum_{p}\frac{a_{p}}{p}\sg^{p}\le\frac12\sum_{p}a_{p}\sg^{p}
=\frac{\sg\Lam}{2}$, since $p\ge2$.
\end{proof}

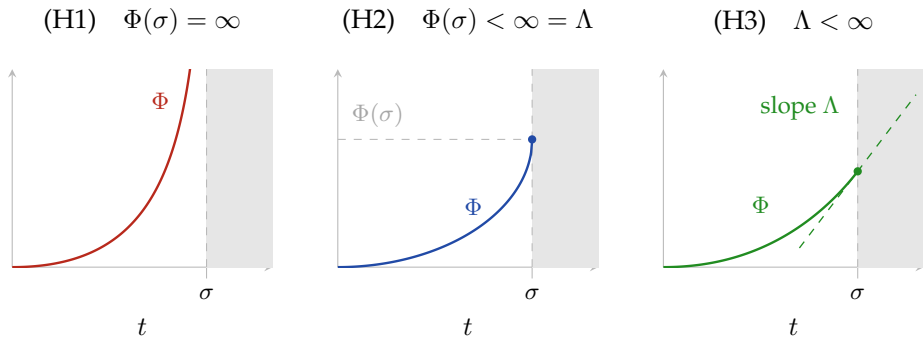
\begin{figure}[ht]
\centering
%---------------------------------------------------------------- (H1)
\begin{tikzpicture}
\begin{axis}[width=.335\textwidth,height=.28\textwidth,
  xmin=0,xmax=1.34,ymin=0,ymax=1.55,
  xtick={1},xticklabels={$\sg$},ytick=\empty,
  xlabel={$t$},title={(H1)\quad$\Phi(\sg)=\infty$}]
\fill[Shade] (axis cs:1,0) rectangle (axis cs:1.34,1.55);
\draw[gray!55,dashed] (axis cs:1,0)--(axis cs:1,1.55);
\addplot[BrickRed,domain=0:0.95,samples=140,smooth] {-ln(1-x)-x};
\node[BrickRed,fig label,anchor=east] at (axis cs:0.86,1.30) {$\Phi$};
\end{axis}
\end{tikzpicture}\hspace{\figgap}
%---------------------------------------------------------------- (H2)
\begin{tikzpicture}
\begin{axis}[width=.335\textwidth,height=.28\textwidth,
  xmin=0,xmax=1.34,ymin=0,ymax=1.55,
  xtick={1},xticklabels={$\sg$},ytick=\empty,
  xlabel={$t$},title={(H2)\quad$\Phi(\sg)<\infty=\Lam$}]
\fill[Shade] (axis cs:1,0) rectangle (axis cs:1.34,1.55);
\draw[gray!55,dashed] (axis cs:1,0)--(axis cs:1,1.55);
\addplot[RoyalBlue,domain=0:0.99999,samples=200,smooth] {1-sqrt(1-x*x)};
\addplot[RoyalBlue,only marks,mark=*,mark size=1.1] coordinates {(1,1)};
\draw[gray!55,dashed] (axis cs:0,1)--(axis cs:1,1);
\node[gray!70,fig label,anchor=south west] at (axis cs:0.02,1.02) {$\Phi(\sg)$};
\node[RoyalBlue,fig label,anchor=north west] at (axis cs:0.60,0.62) {$\Phi$};
\end{axis}
\end{tikzpicture}\hspace{\figgap}
%---------------------------------------------------------------- (H3)
\begin{tikzpicture}
\begin{axis}[width=.335\textwidth,height=.28\textwidth,
  xmin=0,xmax=1.34,ymin=0,ymax=1.55,
  xtick={1},xticklabels={$\sg$},ytick=\empty,
  xlabel={$t$},title={(H3)\quad$\Lam<\infty$}]
\fill[Shade] (axis cs:1,0) rectangle (axis cs:1.34,1.55);
\draw[gray!55,dashed] (axis cs:1,0)--(axis cs:1,1.55);
\addplot[ForestGreen,domain=0:0.9999,samples=200,smooth]
  {x^2-0.25+(1-x)^2/4-(1-x)^2/2*ln(1-x)};
\draw[ForestGreen,dashed,line width=.5pt]
  (axis cs:0.7,0.15)--(axis cs:1.30,1.35);
\addplot[ForestGreen,only marks,mark=*,mark size=1.1] coordinates {(1,0.75)};
\node[ForestGreen,fig label,anchor=east] at (axis cs:0.96,1.25)
  {slope $\Lam$};
\node[ForestGreen,fig label,anchor=south east] at (axis cs:0.60,0.34) {$\Phi$};
\end{axis}
\end{tikzpicture}
\caption{The trichotomy \eqref{eq:trichotomy}, read off the energy density
$\Phi$ near the constraint; the shaded band is the inadmissible range
$t>\sg$. In (H1) the energy is infinite at $\sg$ and the constraint cannot be
attained on a set of positive measure. In (H2), the Born--Infeld and Minkowski
case, the energy is finite but the slope $\Lam=\Phi'(\sg^{-})$ is infinite, and
no obstruction to solvability arises. In (H3) the slope is finite, and this is
the only regime in which solutions cease to exist for large data.}
\label{fig:trichotomy}
\end{figure}

The same trichotomy is displayed in Figure \ref{fig:flux} in terms of the flux
$\varphi$, which is the form in which it will be used throughout: what matters
is whether $\varphi$ blows up at $\sg$ or reaches a finite value $\Lam$.

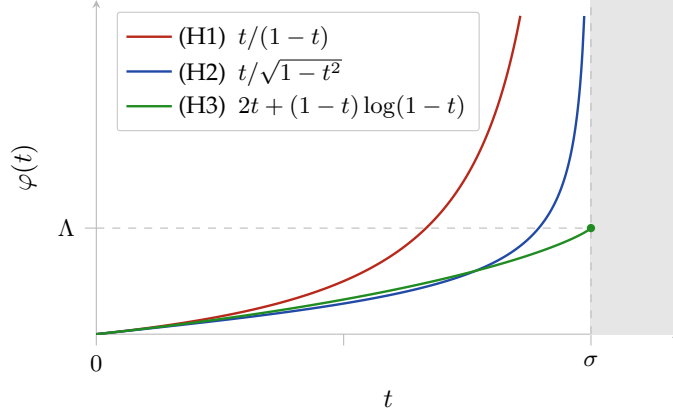
\begin{figure}[ht]
\centering
\begin{tikzpicture}
\begin{axis}[width=.62\textwidth,height=.40\textwidth,
  xmin=0,xmax=1.18,ymin=0,ymax=6.3,
  xtick={0,0.5,1},xticklabels={$0$,,$\sg$},
  ytick={2},yticklabels={$\Lam$},
  xlabel={$t$},ylabel={$\varphi(t)$},
  legend style={at={(0.035,0.965)},anchor=north west},legend cell align=left]
\fill[Shade] (axis cs:1,0) rectangle (axis cs:1.18,6.3);
\draw[gray!55,dashed] (axis cs:1,0)--(axis cs:1,6.3);
\draw[gray!55,dashed] (axis cs:0,2)--(axis cs:1,2);
\addplot[BrickRed,domain=0:0.8572,samples=200,smooth] {x/(1-x)};
\addlegendentry{(H1)\ \ $t/(1-t)$}
\addplot[RoyalBlue,domain=0:0.98639,samples=300,smooth] {x/sqrt(1-x*x)};
\addlegendentry{(H2)\ \ $t/\sqrt{1-t^{2}}$}
\addplot[ForestGreen,domain=0:0.9999,samples=300,smooth]
  {2*x+(1-x)*ln(1-x)};
\addlegendentry{(H3)\ \ $2t+(1-t)\log(1-t)$}
\addplot[ForestGreen,only marks,mark=*,mark size=1.1,forget plot]
  coordinates {(1,2)};
\end{axis}
\end{tikzpicture}
\caption{The flux $\varphi$ for the three model sequences: $a_{p}\equiv1$ (H1);
the Minkowski sequence $a_{2k+2}=4^{-k}\binom{2k}{k}$ (H2); and the logarithmic
sequence of Example \ref{ex:log} (H3), the only one of the three for which
$\varphi(\sg^{-})=\Lam$ is finite; the first two leave the picture from above.}
\label{fig:flux}
\end{figure}

It is convenient to encode the constraint in the integrand. We set
\begin{equation}\label{eq:G}
G(\xi)=
\begin{cases}
\Phi(|\xi|), & |\xi|\le\sg,\\
+\infty, & |\xi|>\sg,
\end{cases}
\qquad \xi\in\R^{N},
\end{equation}
with $\Phi(\sg)$ understood as in \eqref{eq:LamPhi} (possibly $+\infty$). By
Lemma \ref{lem:elementary}, $G:\R^{N}\to[0,+\infty]$ is convex, lower
semicontinuous, even, and strictly convex on the interior of its domain.

\begin{remark}\label{rem:extreme}
If $\sg=0$ the constraint forces $\nb u=0$, so that \eqref{eq:main} admits a
solution only if $g$ is constant and $f=0$: no nontrivial solution exists,
which is the exact analogue of \cite[Remark 2.5]{P2010}. If $\sg=+\infty$
there is no gradient constraint and $\Lam=+\infty$; then $\Phi$ is an
$N$-function which typically fails the $\Delta_{2}$ condition (e.g.\
$a_{p}=1/p!$ gives $\Phi(t)\sim \frac{e^{t}}{t}$), so that the natural functional
framework is a nonreflexive Orlicz--Sobolev space; we refer to
\cite{Gossez,ACCZG} and we do not pursue the analysis of this case here, except to note that
all the results of Section \ref{sec:var} extend, with a different compactness argument,  with $S_{g}$ replaced by
$\{v:\ F(v)<\infty,\ v-\bar g\in W^{1,1}_{0}\}$.
\end{remark}

\subsection{Main instances of the operators involved}

It is natural to ask which quasilinear operators $\dive(\mathcal A(|\nb
u|)\nb u)$ can be written in the form \eqref{eq:main}. The answer is a
classical one.

{Let}  $\mathcal A:[0,\sg)\to\R$ {be a  real analytic function. Then} $\mathcal A$ is the sum of
a power series with nonnegative coefficients if and only if $\mathcal A$ is
\emph{absolutely monotone} on $[0,\sg)$, i.e.\ $\mathcal A^{(k)}(t)\ge0$ for
every $k\ge0$ and every $t\in[0,\sg)$; the coefficients are then
$a_{p+2}=\mathcal A^{(p)}(0)/p!$.

This is Bernstein's characterization of absolutely monotone functions. All the
model operators \eqref{eq:model} are of this type, because
$(1-z^{q})^{-\alpha}$ has nonnegative Taylor coefficients for every
$\alpha>0$, $q\in\N$. On the other hand:

\medskip
The \emph{Euclidean} mean curvature operator, $\mathcal
A(t)=(1+t^{2})^{-1/2}$, is \emph{not} of this form: its Taylor coefficients
have alternating signs. Thus the family \eqref{eq:main} interpolates between
the Laplacian and Lorentzian, gradient constrained operators, and not between
the Laplacian and the minimal surface operator. This is consistent with the
fact that the Euclidean operator has a bounded flux with $\sg=+\infty$, a
combination which \eqref{eq:H} excludes; see however Remark
\ref{rem:euclcheck}.
 \medskip 

\begin{remark}\label{rem:noninteger}
Nothing in the sequel uses that the exponents are integers. All the results
below hold, with identical proofs, for operators of the form
$-\sum_{j}a_{j}\Delta_{p_{j}}u$ with $2\le p_{1}<p_{2}<\cdots\to\infty$ real, undes suitable compatilitity assumptions (e.g. $\log j = o(p_j)$), 
upon replacing \eqref{eq:sigma} by
$\sg^{-1}=\limsup_{j}a_{j}^{1/p_{j}}$, aand even for continuous superpositions
\[
-\int_{2}^{\infty}\Delta_{p}u\,d\mu(p)=-\dive\Big(\Big[\int_{2}^{\infty}|\nb
u|^{p-2}d\mu(p)\Big]\nb u\Big)
\]
with $\mu$ a nonnegative measure on $[2,\infty)$ such that $\mu(\{2\})>0$, upon
replacing \eqref{eq:sigma} by
$$\sg:=\sup\{t\ge0:\ \int_{2}^{\infty}t^{p}\,d\mu(p)<\infty\}.$$  The class of coefficients $\mathcal A$ so obtained, namely the functions
$t\mapsto\int_{2}^{\infty}t^{p-2}\,d\mu(p)$ (Mellin transforms of nonnegative
measures), is strictly larger than the one satisfying Bernstein's criterion, and
contains for instance
\[
\mathcal A(t)=1+t^{r}(1-t^{q})^{-\alpha},\qquad r>0\ \text{noninteger},\ q\in\N,\ \alpha>0,
\]
which corresponds to $\mu=\delta_{2}+\sum_{k\ge0}c_{k}\,\delta_{2+r+qk}$, with
$c_{k}\ge0$ the Taylor coefficients of $(1-z)^{-\alpha}$, and for which $\sg=1$
and $\mathcal A(0)=\mu(\{2\})=1>0$.

\smallskip 
Superpositions of infinitely many operators have recently been the object of a
systematic study in the fractional setting, where it is the \emph{order} of the
operators, rather than their growth exponent, that varies: see \cite{DPSV} for
the superposition $\int_{[0,1]}(-\Delta)^{s}u\,d\mu(s)$ of fractional
Laplacians modulated by a finite measure, \cite{DPLSV} for maximum principles
and spectral analysis, also in the presence of a signed measure, and
\cite{DPLSV2} for the $(s,p)$-superposition of fractional $p$-Laplacians, which
is the closest in spirit to the present paper. The mechanism there is, however,
structurally different from ours: the superposition is governed by the
dominance of the highest order appearing in the support of $\mu$, and no
phenomenon of the type of the radius of convergence \eqref{eq:sigma}{, hence
no gradient constraint, occurs}.
\end{remark}

\subsection{The constraint, the admissible set and the notion of solution}

The following elementary proposition is the reason why the problem is a
constrained one; it is the exact counterpart of \cite[Proposition 2.3]{P2010},
with $u\le\sg$ replaced by $|\nb u|\le\sg$.

\begin{proposition}\label{prop:constraint}
Let $u\in W^{1,1}_{\rm loc}(\Om)$ be such that
$\sum_{p\ge2}a_{p}\int_{\Om}|\nb u|^{p}<\infty$. Then $\|\nb
u\|_{L^{\infty}(\Om)}\le\sg$.
\end{proposition}

\begin{proof}
Let $C=\sum_{p\ge2}a_{p}\int_{\Om}|\nb u|^{p}<\infty$ denote the value of the series. For every $p$ with $a_{p}>0$ we have
$\|\nb u\|_{L^{p}(\Om)}\le(C/a_{p})^{1/p}$. Choose a subsequence $p_{j}\to
\infty$ realizing the $\limsup$ in \eqref{eq:sigma}, i.e.\
$a_{p_{j}}^{1/p_{j}}\to1/\sg$. Then
$\|\nb u\|_{L^{p_{j}}(\Om)}\le C^{1/p_{j}}a_{p_{j}}^{-1/p_{j}}\to\sg$, and
since $\|\nb u\|_{L^{p}(\Om)}\to\|\nb u\|_{L^{\infty}(\Om)}$ as $p\to\infty$, the conclusion follows.
\end{proof}

Let now $g\in C^{0,1}(\partial\Om)$ and let
\begin{equation}\label{eq:L}
L:=\Lip(g,\partial\Om)=\sup_{x\neq y\in\partial\Om}\frac{|g(x)-g(y)|}{|x-y|} .
\end{equation}
The McShane extension $\bar g(x)=\min_{y\in\partial\Om}\{g(y)+L|x-y|\}$
satisfies $\bar g=g$ on $\partial\Om$ and $\Lip(\bar g,\R^{N})=L$; one may
also use the absolutely minimizing Lipschitz extension of \cite{ACJ04}. We set
$S_{g}$ as the convex set
\begin{equation}\label{eq:Sg}
S_{g}:=\Big\{v\in C^{0,1}(\overline\Om):\ \|\nb v\|_{L^{\infty}(\Om)}\le\sg,
\quad v=g\ \text{on }\partial\Om\Big\},
\end{equation}
which is nonempty as soon as $L\le\sg$, and we observe that for
every $v\in S_{g}$
\begin{equation}\label{eq:Linfbound}
\|v\|_{L^{\infty}(\Om)}\le M_{0}:=\|g\|_{L^{\infty}(\partial\Om)}+\sg\,
d_{\Om},
\end{equation}
$d_{\Om}$ being the intrinsic diameter of $\Om$. Compatibility of the boundary
datum with the constraint, i.e.\ $L\le\sg$, is a genuine restriction; in the
Minkowski setting ($\sg=1$) it is precisely the requirement that  guarantees that $g$ admits a weakly spacelike extension, cf.\ \cite{BS1982}.

\begin{definition}\label{def:weak}
A function $u\in S_{g}$ is a \emph{weak solution} of \eqref{eq:main} if
\begin{equation}\label{eq:weak}
\sum_{p=2}^{\infty}a_{p}\int_{\Om}|\nb u|^{p-2}\nb u\cdot\nb\psi
=\int_{\Om}f\psi,
\end{equation}
\text{for every }$\psi\in W^{1,\infty}_{0}(\Om)$ such that all the terms are absolutely convergent.
\end{definition}

\begin{remark}{\label{rem:quantifier}
The class of test functions in Definition \ref{def:weak} depends on $u$ itself.
This is deliberate: by Remark \ref{rem:welldef} below the restriction is empty
both when $\|\nb u\|_{\infty}<\sg$ and when $\Lam<\infty$, which are the only two
situations in which Definition \ref{def:weak} is used in the sequel, while in the
intermediate regime (H2) it avoids requiring a convergence that no a priori
estimate guarantees.}
\end{remark}

\begin{remark}\label{rem:welldef}
If $\|\nb u\|_{\infty}\le\gamma<\sg$ then \eqref{eq:weak} is automatically
absolutely convergent, since
$\sum_{p}a_{p}\int|\nb u|^{p-1}|\nb\psi|\le|\Om|\,\|\nb\psi\|_{\infty}\varphi(\gamma)<\infty$.
If $\Lam<\infty$, then \eqref{eq:weak} is absolutely convergent for every
$u\in S_{g}$, because $|\aaa(\nb u)|=\varphi(|\nb u|)\le\Lam$ a.e. It is only
in the regime (H2), $\Phi(\sg)<\infty=\Lam$, that the left hand side of
\eqref{eq:weak} may fail to make sense for some $u\in S_{g}$.
\end{remark}

\section{The variational problem}\label{sec:var}

Throughout this section $f\in L^{1}(\Om)$ and $g\in C^{0,1}(\partial\Om)$ with
$L\le\sg$. We consider the energy
\begin{equation}\label{eq:F}
F(v)=\int_{\Om}G(\nb v)-\int_{\Om}fv
=\sum_{p=2}^{\infty}\frac{a_{p}}{p}\int_{\Om}|\nb v|^{p}-\int_{\Om}fv,
\qquad v\in S_{g},
\end{equation}
whose formal Euler--Lagrange equation is precisely \eqref{eq:main}. Because of
\eqref{eq:Linfbound}, the linear part of $F$ is finite for every $f\in
L^{1}(\Om)$ (indeed for every Radon measure), and $F$ is bounded from below, i.e. 
\begin{equation}\label{eq:below}
F(v)\ \ge\ -\|f\|_{L^{1}(\Om)}M_{0}\qquad\text{for every }v\in S_{g}.
\end{equation}
  The only nontrivial requirement is
that $F$ be not identically $+\infty$ on $S_{g}$, that is 

\begin{assumption}\label{ass:A}
There exists $\bar g\in S_{g}$ with $F(\bar g)<+\infty$, i.e.\ with
$\int_{\Om}\Phi(|\nb\bar g|)<+\infty$.
\end{assumption}

\begin{proposition}\label{prop:Aholds}
Assumption \ref{ass:A} holds in each of the following cases:
\begin{enumerate}
\item[(i)] $L<\sg$ (any Lipschitz extension of $g$ with Lipschitz constant $L$ will do the job, since
then $\int_{\Om}\Phi(|\nb\bar g|)\le|\Om|\Phi(L)<\infty$);
\item[(ii)] $L\le\sg$ and $\Phi(\sg)<\infty$ (regimes (H2) and (H3));
\item[(iii)] $L=\sg$, $\Phi(\sg)=\infty$, and there exist an extension
$\bar g$ of $g$ with $\|\nb\bar g\|_{\infty}\le\sg$, a modulus
$\omega$ with $|\{|\nb\bar g|>\delta\}|\le\omega(\sg-\delta)$, and a sequence
$\delta_{p}\uparrow\sg$ such that
\begin{equation}\label{eq:deltap}
\sum_{p\ge2}\frac{a_{p}}{p}\Big[\delta_{p}^{p}\,|\Om|
+\sg^{p}\,\omega(\sg-\delta_{p})\Big]<\infty .
\end{equation}
\end{enumerate}
\end{proposition}

\begin{proof}
Only (iii) requires a comment: splitting $\Om$ into
$A_{p}=\{|\nb\bar g|>\delta_{p}\}$ and its complement, we have
\[
\int_{\Om}\Phi(|\nb\bar g|)
=\sum_{p}\frac{a_{p}}{p}\Big[\int_{A_{p}}|\nb\bar g|^{p}
+\int_{\Om\setminus A_{p}}|\nb\bar g|^{p}\Big]
\le\sum_{p}\frac{a_{p}}{p}\big[\sg^{p}|A_{p}|+\delta_{p}^{p}|\Om|\big],
\]
which is finite by \eqref{eq:deltap}.
\end{proof}

\begin{remark}\label{rem:Afails}
Assumption \ref{ass:A} may genuinely fail in the borderline case $L=\sg$,
$\Phi(\sg)=\infty$. Take $N=1$, $\Om=(0,1)$, $g(0)=0$, $g(1)=\sg$: then
$L=\sg$ and $S_{g}$ reduces to the single function $v(x)=\sg x$, for which
$F(v)=\Phi(\sg)-\int fv=+\infty$. Hence no minimizer exists. This shows that
some hypothesis of the type of Proposition \ref{prop:Aholds}(iii) is
unavoidable, and that the geometry of the set where the extension saturates
the constraint is what matters.
\end{remark}

\begin{theorem}\label{thm:min}
Let \eqref{eq:H} hold with $0<\sg<\infty$, let $f\in L^{1}(\Om)$ and let
Assumption \ref{ass:A} be satisfied. Then $F$ has a unique minimizer $u$ in
$S_{g}$.
\end{theorem}

\begin{proof}
  By \eqref{eq:below} and Assumption \ref{ass:A} we have
$-\infty<\inf_{S_{g}}F\le F(\bar g)<+\infty$. Let $\{u_{n}\}\subset S_{g}$ be
a minimizing sequence. Since $\|\nb u_{n}\|_{\infty}\le\sg$ and, by
\eqref{eq:Linfbound}, $\|u_{n}\|_{\infty}\le M_{0}$, the sequence is bounded
and equi-Lipschitz; by the Ascoli--Arzel\`a theorem, up to a subsequence,
$u_{n}\to u$ uniformly on $\overline\Om$ and $\nb u_{n}\rightharpoonup\nb u$
weakly-$*$ in $L^{\infty}(\Om;\R^{N})$. Consequently $u=g$ on $\partial\Om$
and $\|\nb u\|_{\infty}\le\liminf\|\nb u_{n}\|_{\infty}\le\sg$, that is $u\in
S_{g}$. For every fixed $p$, by weak lower semicontinuity of the convex
functional $v\mapsto\int_{\Om}|\nb v|^{p}$,
\[
\int_{\Om}|\nb u|^{p}\le\liminf_{n}\int_{\Om}|\nb u_{n}|^{p} ,
\]
so that, by Fatou's lemma applied with respect to the counting measure in the
variable $p$,
\[
\sum_{p}\frac{a_{p}}{p}\int_{\Om}|\nb u|^{p}
\le\sum_{p}\liminf_{n}\frac{a_{p}}{p}\int_{\Om}|\nb u_{n}|^{p}
\le\liminf_{n}\sum_{p}\frac{a_{p}}{p}\int_{\Om}|\nb u_{n}|^{p} .
\]
Since $u_{n}\to u$ uniformly and $f\in L^{1}$, we get $\int fu_{n}\to\int fu$
and therefore $F(u)\le\liminf_{n}F(u_{n})=\inf_{S_{g}}F$.

Uniqueness follows by contradiction from the fact that the map
$\xi\mapsto G(\xi)$ is strictly convex on $\{|\xi|\le\sg\}$: for $\xi\neq\eta$
one has $|\tfrac{\xi+\eta}{2}|\le\tfrac{|\xi|+|\eta|}{2}$, with strict
inequality unless $\xi$ and $\eta$ are parallel and equally oriented, and one
concludes by the strict convexity and the strict monotonicity of $\Phi$. Two
distinct minimizers of finite energy would then satisfy
$F(\tfrac{u+v}{2})<\inf_{S_{g}}F$.
\end{proof}

\begin{remark}
The proof used neither the sign of $f$ nor any summability beyond $L^{1}$, and
it did not use Assumption \ref{ass:A} except to guarantee
$\inf_{S_{g}}F<\infty$; in particular it covers all three regimes
\eqref{eq:trichotomy} simultaneously. Compare with  \cite{BS1982}, where the variational problem is set among weakly spacelike functions, while smooth solutions are obtained in  \cite[Theorem 3.6]{BS1982} under a strictly spacelike extension of the datum. 
\end{remark}

\subsection{The variational inequality}

The next results aim to  give a  general statement relating  the minimizer to
problem \eqref{eq:main} and a variational inequality that  holds in all regimes and requires no a priori
estimate whatsoever. First we need the following technical lemma:  

\begin{lemma}[the pointwise directional derivative]\label{lem:dirder}
Let $\xi,\zeta\in\R^{N}$ with $|\xi|\le\sg$, $|\zeta|\le\sg$ and $G(\xi)<\infty$ (where $G$ is defined in \eqref{eq:G}),
and set $\eta=\zeta-\xi$. Then $t\mapsto t^{-1}\big[G(\xi+t\eta)-G(\xi)\big]$ is
nondecreasing on $(0,1]$ and its limit as $t\downarrow0$, which we denote by
$\delta(\xi,\zeta)\in[-\infty,+\infty)$, is given by
\[
\delta(\xi,\zeta)=
\begin{cases}
\aaa(\xi)\cdot\eta, & |\xi|<\sg,\\[1mm]
\Lam\,\sg^{-1}\,\xi\cdot\eta, & |\xi|=\sg,\ \Lam<\infty,\\[1mm]
0, & |\xi|=\sg,\ \Lam=\infty,\ \zeta=\xi,\\[1mm]
-\infty, & |\xi|=\sg,\ \Lam=\infty,\ \zeta\neq\xi.
\end{cases}
\]
Moreover $2\,\xi\cdot\eta\le-|\eta|^{2}$ when $|\xi|=\sg$, so that
$\delta(\xi,\zeta)\le0$ there.
\end{lemma}

\begin{proof}
Monotonicity of the difference {quotient} readily follows by  convexity of $G$. If $|\xi|<\sg$ then
$G$ is differentiable at $\xi$ with $\nb G(\xi)=\aaa(\xi)$, whence the first
case. Let $|\xi|=\sg$; from $|\zeta|^{2}=|\xi|^{2}+2\xi\cdot\eta+|\eta|^{2}
\le\sg^{2}=|\xi|^{2}$ we get $2\xi\cdot\eta+|\eta|^{2}\le0$, and therefore, for
$t\in(0,1)$,
\[
|\xi+t\eta|^{2}=\sg^{2}+t\big(2\xi\cdot\eta+t|\eta|^{2}\big)
\le\sg^{2}-t(1-t)|\eta|^{2},
\qquad\text{so}\qquad
|\xi+t\eta|\le\sg-\frac{t(1-t)|\eta|^{2}}{2\sg} .
\]
In particular $|\xi+t\eta|<\sg$ for $0<t<1$ and $\eta\neq0$, so that
$G(\xi+t\eta)=\Phi(|\xi+t\eta|)$ and the quotient is
$t^{-1}[\Phi(|\xi+t\eta|)-\Phi(\sg)]$. If $\Lam<\infty$ then
$\Phi\in C^{1}([0,\sg])$ with $\Phi'(\sg)=\Lam$, and the chain rule gives
$\delta=\Lam\,\frac{d}{dt}|\xi+t\eta|\big|_{t=0^{+}}=\Lam\,\sg^{-1}\xi\cdot\eta$.
If $\Lam=\infty$ and $\eta\neq0$, put $c=|\eta|^{2}/(2\sg)>0$; for $t\le\frac12$
the previous  bound gives $|\xi+t\eta|\le\sg-ct/2$, hence
\[
\frac{\Phi(|\xi+t\eta|)-\Phi(\sg)}{t}
\le-\frac{c}{2}\cdot\frac{\Phi(\sg)-\Phi(\sg-ct/2)}{ct/2}
\ \longrightarrow\ -\frac{c}{2}\,\Phi'(\sg^{-})=-\infty .
\]
The case $\eta=0$ is trivial.
\end{proof}

\begin{theorem}\label{thm:VI}
Under the assumptions of Theorem \ref{thm:min}, let $u$ be the minimizer and set
$\Sigma=\{x\in\Om:\ |\nb u(x)|=\sg\}$. Let $v\in S_{g}$ with $F(v)<\infty$ and
write $\delta_{v}(x)=\delta\big(\nb u(x),\nb v(x)\big)$ as in Lemma
\ref{lem:dirder}. Then:
\begin{enumerate}
\item[(i)] if $\Lam=\infty$, then $\nb v=\nb u$ a.e.\ on $\Sigma$;
\item[(ii)] $\delta_{v}^{-}\in L^{1}(\Om)$ and
$\displaystyle\int_{\Om}\delta_{v}\ \ge\ \int_{\Om}f\,(v-u)$;
\item[(iii)] the partial sums of the series in \eqref{eq:VI} converge in
$(-\infty,+\infty]$ to $\int_{\Om}\delta_{v}$, so that
\begin{equation}\label{eq:VI}
\sum_{p=2}^{\infty}a_{p}\int_{\Om}|\nb u|^{p-2}\nb u\cdot\nb(v-u)\ \ge\
\int_{\Om}f\,(v-u).
\end{equation}
\end{enumerate}
\end{theorem}

\begin{proof}
Write $w=v-u$ and $u_{t}=u+tw\in S_{g}$, $t\in[0,1]$.

\emph{Step 1 (minimality).} The function $t\mapsto F(u_{t})$ is convex on
$[0,1]$, finite at $t=0$ and $t=1$, and attains its minimum at $t=0$. Hence, for
every $t\in(0,1]$,
\begin{equation}\label{eq:step1}
\int_{\Om}Q_{t}\ \ge\ \int_{\Om}fw,\qquad
Q_{t}:=\frac{G(\nb u+t\nb w)-G(\nb u)}{t},
\end{equation}
where $\int_{\Om}fw$ is finite because $f\in L^{1}(\Om)$ and
$\|w\|_{\infty}\le2M_{0}$ by \eqref{eq:Linfbound}.

\emph{Step 2 (passage to the limit).} By Lemma \ref{lem:dirder}, applied
pointwise with $\xi=\nb u(x)$ and $\zeta=\nb v(x)$, the family $Q_{t}$ is
nondecreasing in $t$ and $Q_{t}\downarrow\delta_{v}$ a.e.\ as $t\downarrow0$;
moreover $Q_{t}\le Q_{1}=G(\nb v)-G(\nb u)\in L^{1}(\Om)$, since
$F(u),F(v)<\infty$. Applying the monotone convergence theorem to the
nonnegative nondecreasing family $Q_{1}-Q_{t}$ we obtain
\[
\int_{\Om}Q_{t}\ \longrightarrow\ \int_{\Om}\delta_{v}\ \in[-\infty,+\infty)
\qquad\text{as }t\downarrow0 .
\]

\emph{Step 3 (conclusion, and integrability of the negative part).} Passing to
the limit in \eqref{eq:step1} gives
$\int_{\Om}\delta_{v}\ge\int_{\Om}fw>-\infty$. Since
$\delta_{v}^{+}\le Q_{1}^{+}\in L^{1}(\Om)$, this is precisely the statement
that $\delta_{v}^{-}\in L^{1}(\Om)$, and (ii) is proved.  

\emph{Step 4 (proof of (i)).} If $\Lam=\infty$ and the set
$\{x\in\Sigma:\nb v(x)\neq\nb u(x)\}$ had positive measure, then
$\delta_{v}=-\infty$ there by Lemma \ref{lem:dirder}, whence
$\int_{\Om}\delta_{v}=-\infty$, contradicting (ii).

\emph{Step 5 (identification of the series).} All the terms of the series
vanish on $\{\nb u=0\}$. On $\{\nb u\neq0\}$ set
$s=|\nb u|^{-1}\,\nb u\cdot\nb w$, so that the $n$-th partial sum equals
$\int_{\Om}\varphi_{n}(|\nb u|)\,s$ with
$\varphi_{n}(t)=\sum_{p\le n}a_{p}t^{p-1}\uparrow\varphi(t)$. On $\{s\ge0\}$ the
integrands increase and on $\{s<0\}$ they decrease and are bounded above by $0$;
by monotone convergence on each of the two sets separately, the partial sums
converge to $\int_{\Om}\varphi(|\nb u|)\,s$, the limit being $>-\infty$ by (ii).
Finally, $\varphi(|\nb u|)s=\delta_{v}$ a.e.: this is clear off $\Sigma$, on
$\Sigma$ it is the second case of Lemma \ref{lem:dirder} when $\Lam<\infty$, and
when $\Lam=\infty$ one has $s=0$ a.e.\ on $\Sigma$ by (i), so that both sides
vanish there.
\end{proof}

\begin{corollary}\label{cor:VI2PDE}
Let $u$ be the minimizer. If $\|\nb u\|_{L^{\infty}(\Om)}<\sg$, then $u$ is a
weak solution of \eqref{eq:main} in the sense of Definition \ref{def:weak}.
Conversely, if $u\in S_{g}$ is a weak solution of \eqref{eq:main} and
$\int_{\Om}\aaa(\nb u)\cdot\nb(v-u)$ is well defined for every $v\in S_{g}$
with $F(v)<\infty$, then $u$ is the minimizer of $F$ on $S_{g}$.
\end{corollary}

\begin{proof}
If $\gamma:=\|\nb u\|_{\infty}<\sg$, then for $\psi\in W^{1,\infty}_{0}(\Om)$
and $|t|$ small the function $u+t\psi$ still belongs to $S_{g}$; since $G$ is
$C^{1}$ on $\{|\xi|\le\frac{\gamma+\sg}{2}\}$ with
$|\nb G|=\varphi\le\varphi(\frac{\gamma+\sg}{2})<\infty$ there, one may
differentiate under the integral sign and obtain \eqref{eq:weak}; absolute
convergence is given by Remark \ref{rem:welldef}. Conversely, by convexity of $G$,
\[
F(v)-F(u)\ \ge\ \int_{\Om}\aaa(\nb u)\cdot\nb(v-u)-\int_{\Om}f(v-u)=0
\]
for every admissible $v$, using \eqref{eq:weak} with $\psi=v-u\in
W^{1,\infty}_{0}(\Om)$.
\end{proof}

\begin{remark}\label{rem:H1ae}
In the regime (H1), $\Phi(\sg)=\infty$, the minimizer satisfies
\begin{equation}\label{eq:ae}
|\nb u|<\sg\quad\text{a.e. in }\Om ,
\end{equation}
since otherwise $F(u)=+\infty$, contradicting $F(u)\le F(\bar g)<\infty$. We
stress that \eqref{eq:ae} is \emph{strictly weaker} than
$\|\nb u\|_{\infty}<\sg$ and is not sufficient to apply Corollary
\ref{cor:VI2PDE}: for instance $u(x)=\sg\int_{0}^{x}(1-s)^{1/2}\,ds$ on
$(0,1)$ satisfies $|u'|<\sg$ {at every point of $(0,1)$, the value $\sg$ being
approached only as $x\to0^{+}$,} but $\|u'\|_{\infty}=\sg$. Obtaining a
quantitative bound $\|\nb u\|_{\infty}\le\gamma<\sg$ is the only missing step
in the whole theory when $\Lam=\infty$; this is the object of the next
section.
\end{remark}

We also  have the following consequence which is relevant in the Born--Infeld regime (H2): 

\begin{corollary}\label{cor:H2nosat}
Assume $\Lam=\infty$ and $L<\sg$, and let $u$ be the minimizer. Then
\[
|\nb u|<\sg\qquad\text{a.e. in }\Om .
\]
More generally, without assuming $L<\sg$, one has $\nb u=\nb v$ a.e.\ on
$\Sigma$ for every $v\in S_{g}$ with $F(v)<\infty$; in particular $\nb u=\nb\bar
g$ a.e.\ on $\Sigma$ for every admissible extension $\bar g$ of the boundary
datum.
\end{corollary}

\begin{proof}
The second assertion is Theorem \ref{thm:VI}(i). For the first, let $\bar g$ be
a Lipschitz extension of $g$ of rank $L$, so that $|\nb\bar g|\le L<\sg$ a.e.\
and $F(\bar g)\le|\Om|\Phi(L)+\|f\|_{L^{1}}M_{0}<\infty$. On $\Sigma$ one has
$|\nb u|=\sg>L\ge|\nb\bar g|$, hence $\nb u\neq\nb\bar g$ there; by Theorem
\ref{thm:VI}(i) applied to $v=\bar g$ this forces $|\Sigma|=0$.
\end{proof}

\begin{remark}\label{rem:rigidity}
Corollary \ref{cor:H2nosat} is the analogue, in the regime (H2), of Remark
\ref{rem:H1ae}, and the two have different mechanisms: in (H1) the constraint
cannot be attained on a set of positive measure because the energy would be
infinite there, whereas in (H2) the energy at the constraint is finite and it is
the \emph{infinite slope} $\Phi'(\sg^{-})=\Lam=\infty$ that forbids it.
Both statements give $|\nb u|<\sg$ a.e.\ and neither gives
$\|\nb u\|_{\infty}<\sg$, which remains out of reach in general (Remark
\ref{rem:H1ae}).

 {The mechanism behind Corollary \ref{cor:H2nosat} is the one of Step 4 in
the proof of Theorem \ref{thm:VI}: if the constraint were active on a set of
positive measure, moving from $u$ towards any admissible competitor $v$ would
decrease the energy at rate $\int_{\Om}\delta_{v}=-\infty$, which is
incompatible with minimality. It breaks down precisely when $\Lam<\infty$,
because there the directional derivative on $\Sigma$ equals
$\Lam\sg^{-1}\nb u\cdot\nb(v-u)$ by Lemma \ref{lem:dirder}, a finite and
nonpositive quantity: no contradiction arises, and indeed saturation does occur
in the regime (H3), see Theorem \ref{thm:radial}. It also breaks down in the
borderline case $L=\sg$, where every admissible competitor may have the same
gradient as $u$ on $\Sigma$, so that $\delta_{v}=0$ there; this is exactly the
regime in which the maximal surfaces of \cite{BS1982} are allowed to contain
light rays; sharp conditions on the prescribed curvature ruling them out have
been recently obtained in \cite{DFM2026}.}\end{remark}

\section{Lipschitz estimates and existence of solutions when
\texorpdfstring{$\Lam=\infty$}{Lambda=infinity}}\label{sec:existence}

\subsection{A comparison lemma}

The a priori bound $\|\nb u\|_{\infty}\le\gamma<\sg$ cannot be obtained from the
regularity theory for functionals with $(p,q)$ or nonstandard growth
\cite{M2006,M2020,ELM2004,CM2015,BCM2018,L1991}, which requires a $\Delta_{2}$
condition on the integrand: here $G$ is not even finite outside
$\{|\xi|\le\sg\}$. We use instead the following comparison principle, which
relies only on convexity and on the lattice identity, and is therefore
insensitive to the growth of $G$.

\begin{lemma}\label{lem:comparison}
Let $\omega\subset\R^{N}$ be a bounded open set, $\lambda\in\R$, and let
$u,w\in C^{0,1}(\overline\omega)$ be such that each of them minimizes
\[
\mathcal F_{\omega}(v)=\int_{\omega}G(\nb v)-\lambda\int_{\omega}v
\]
among the functions of $C^{0,1}(\overline\omega)$ with its own boundary values,
with $\mathcal F_{\omega}(u),\mathcal F_{\omega}(w)<\infty$. Then
\[
\max_{\overline\omega}(w-u)=\max_{\partial\omega}(w-u).
\]
\end{lemma}

\begin{proof}
Let $c=\max_{\partial\omega}(w-u)$ and $\tilde w=w-c$, which still minimizes
$\mathcal F_{\omega}$ among functions with its own boundary values (adding a
constant shifts $\mathcal F_{\omega}$ by the constant $-\lambda c|\omega|$ for
all competitors). Set
\[
m(x)=\min(u,\tilde w),\qquad M(x)=\max(u,\tilde w) \ \ \text{for }\ x\in\omega.  
\]
Both belong to $C^{0,1}(\overline\omega)$, and since $\tilde w\le u$ on
$\partial\omega$ we have $m=\tilde w$ and $M=u$ on $\partial\omega$. Hence $M$
is admissible in the  variational  problem solved by $u$, and $m$ is admissible in
the one solved by $\tilde w$:
\begin{equation}\label{mM}
\mathcal F_{\omega}(M)\ge\mathcal F_{\omega}(u),\qquad
\mathcal F_{\omega}(m)\ge\mathcal F_{\omega}(\tilde w) .
\end{equation}
Observe that, setting $P=\{u>\tilde w\}$ one has $(m,M)=(\tilde w,u)$ on $P$
and $(m,M)=(u,\tilde w)$ on $\omega\setminus P${, so that $m$ and $M$ are a
pointwise relabeling of the pair $(u,\tilde w)$. The same relabeling holds
a.e.\ for the gradients: on the open sets $\{u>\tilde w\}$ and $\{u<\tilde w\}$
this is clear, while on the contact set $\{u=\tilde w\}$ one has
$\nb u=\nb\tilde w$ a.e., since a Sobolev function has vanishing gradient a.e.\
on any of its level sets, applied here to $u-\tilde w\in W^{1,\infty}(\omega)$}.
Splitting each integral over $P$
and its complement, and using $m+M=u+\tilde w$ pointwise for the lower order
term, we thus obtain
\[
\mathcal F_{\omega}(m)+\mathcal F_{\omega}(M)
=\mathcal F_{\omega}(u)+\mathcal F_{\omega}(\tilde w),
\]
 {an identity which uses no property of $G$ whatsoever, neither convexity nor
finiteness, but only the additivity of the integral. Hence the inequalities} \eqref{mM} are
equalities. By the uniqueness of the
minimizer with given boundary values (strict convexity of $G$, as in the proof
of Theorem \ref{thm:min}) we conclude $M=u$, i.e.\ $\tilde w\le u$ in
$\omega$, which implies the claim.
\end{proof}

\begin{remark}{\label{rem:comparisonSobolev}
Let us stress that the proof of Lemma \ref{lem:comparison} uses only two facts: that $m$ and $M$ are
admissible competitors for the two minimum problems, and the additivity of the
integral. Both are available in a Sobolev setting, so that the statement holds
verbatim for $u,w\in W^{1,r}(\omega)$, $r\ge1$, minimizing among the functions
with the same trace, provided the maximum over $\partial\omega$ is understood as
$\inf\{c\in\R:(w-u-c)^{+}\in W^{1,r}_{0}(\omega)\}$ and the maximum over
$\overline\omega$ as an essential supremum. We shall use the lemma in this form
for the truncated problems of Theorem \ref{thm:bsc}, whose minimizers are not
known to be Lipschitz beforehand.}
\end{remark}

\begin{remark}\label{rem:cellina}
It is worth making explicit where exactly the constraint obstructs the passage
from the minimizer to the equation.

Encoding the constraint in the integrand, as we did in \eqref{eq:G}, is not a
technical device: as observed by Cellina \cite{Cellina}, a gradient constrained
problem is \emph{equivalent} to an unconstrained minimization for an
extended-real-valued integrand, and it is in this form that it should be read.
Writing $G=\Phi(|\cdot|)+I_{B_{\sg}}$, with $I_{B_{\sg}}$ the indicator function
of the ball $\{|\xi|\le\sg\}$, one has
\[
\partial G(\xi)=\aaa(\xi)+N_{B_{\sg}}(\xi),
\]
$N_{B_{\sg}}$ being the normal cone. Away from the constraint the second term
vanishes and $\partial G(\xi)=\{\aaa(\xi)\}$ is a single vector; on the
constraint, $|\xi|=\sg$, the normal cone is the outward ray $\{s\xi:s\ge0\}$
and
\[
\partial G(\xi)=\Big\{t\,\xi\ :\ t\ge\frac{\Lam}{\sg}\Big\}
\quad\text{if }\Lam<\infty,
\qquad
\partial G(\xi)=\emptyset\quad\text{if }\Lam=\infty,
\]
the latter because the graph of $\Phi$ reaches $\sg$ with vertical tangent and
admits no supporting hyperplane there.

A formal consequence is the following. Optimality of $u$ means that some measurable
selection $z(x)\in\partial G(\nb u(x))$ satisfies $-\dive z=f$, and off the
saturation set $\Sigma=\{|\nb u|=\sg\}$ the selection is forced to be
$z=\aaa(\nb u)$. On $\Sigma$, instead, $z=(\Lam/\sg+s)\nb u$ with a free
multiplier $s\ge0$, whereas Definition \ref{def:weak} prescribes the specific
field $\aaa(\nb u)$, that is the choice $s\equiv0$. \emph{The minimizer solves
the equation precisely when the multiplier can be taken to vanish on all of
$\Sigma$}, which is why Corollary \ref{cor:VI2PDE} assumes
$\|\nb u\|_{\infty}<\sg$, and which is the gap between the variational
inequality \eqref{eq:VI} and the equation \eqref{eq:weak}. It also explains how
a minimizer can exist while a solution does not: on $\Sigma$ the optimal field
has $|z|=\Lam+s\sg\ge\Lam$, so it may exceed the maximal flux that Theorem
\ref{thm:nonexistence} allows a solution to have.  

Finally, none of this discussion affects Lemma \ref{lem:comparison}, which uses no
differentiability of $G$ whatsoever, and this is what makes it available in a
setting where the regularity theory for $(p,q)$ and nonstandard growth
\cite{M2006,M2020,ELM2004,CM2015,BCM2018,L1991} is not.
\end{remark}

\subsection{The global Lipschitz constant is a boundary quantity}

\begin{theorem}\label{thm:haarrado}
Let $f\equiv\lambda$ be constant, let Assumption \ref{ass:A} hold and let $u$
be the minimizer of $F$ on $S_{g}$. Define the \emph{boundary--to--interior
slope}
\begin{equation}\label{eq:S}
S:=\sup\Big\{\frac{|u(x)-g(y)|}{|x-y|}\ :\ x\in\Om,\ y\in\partial\Om,\
x\neq y\Big\} .
\end{equation}
Then,
\[
\Lip(u,\Om)=S .
\]
In particular, if $S<\sg$ then $u$ is a weak solution of \eqref{eq:main}.
\end{theorem}

\begin{proof}
The inequality $S\le\Lip(u,\Om)$ is trivial, since $u=g$ on $\partial\Om$.
For the converse, fix $\tau\in\R^{N}\setminus\{0\}$ and set
$\Om_{\tau}=\Om-\tau$, $\Om'=\Om\cap\Om_{\tau}$, which is nonempty precisely
when $\tau=y-x$ for some $x,y\in\Om$, and $u_{\tau}(x)=u(x+\tau)$. By locality, $u$ minimizes
$\mathcal F_{\Om'}$ among functions with its own boundary values on $\Om'$;
by the translation invariance of $G$ and of the constant $\lambda$, the same is
true of $u_{\tau}$. If $x\in\partial\Om'$, then either $x\in\partial\Om$, and
then $u_{\tau}(x)-u(x)=u(x+\tau)-g(x)\le S|\tau|$ by \eqref{eq:S}; or
$x+\tau\in\partial\Om$, and then $u_{\tau}(x)-u(x)=g(x+\tau)-u(x)\le S|\tau|$,
again by \eqref{eq:S}; note that $S\ge L$, since every point of $\partial\Om$ is
a limit of points of $\Om$, so that \eqref{eq:S} also bounds
$|g(x+\tau)-g(x)|$ by $S|\tau|$ when both points lie on $\partial\Om$. Lemma
\ref{lem:comparison} then gives
\[
u(x+\tau)-u(x)\le S|\tau|\qquad\text{for every }x\in\Om',
\]
that is, for every $x$ such that both $x$ and $x+\tau$ belong to $\Om$.
Exchanging the roles of $u$ and $u_{\tau}$ yields the same bound for the
absolute value, and letting $\tau=y-x$ we obtain $|u(y)-u(x)|\le S|y-x|$ for
all $x,y\in\Om$. The last assertion follows from Corollary \ref{cor:VI2PDE}.
\end{proof}

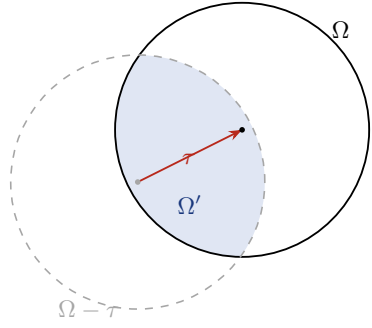
\begin{figure}[ht]
\centering
\begin{tikzpicture}[scale=1.5]
\def\bloba{(0,0) circle (1.12)}
\def\blobb{(-0.92,-0.46) circle (1.12)}
\begin{scope}
  \clip \bloba;
  \fill[RoyalBlue!14] \blobb;
\end{scope}
\draw[black,line width=.7pt] \bloba;
\draw[gray!70,dashed,line width=.6pt] \blobb;
\draw[-{Stealth[length=5pt]},BrickRed,line width=.7pt] (-0.92,-0.46) -- (0,0);
\node[BrickRed,fig label,anchor=north west] at (-0.62,-0.12) {$\tau$};
\node[fig label,anchor=south west] at (0.70,0.72) {$\Om$};
\node[gray!70,fig label,anchor=north] at (-1.35,-1.42) {$\Om-\tau$};
\node[RoyalBlue!75!black,fig label] at (-0.46,-0.66) {$\Om'$};
\fill (0,0) circle (.7pt); \fill[gray!70] (-0.92,-0.46) circle (.7pt);
\end{tikzpicture}
\caption{The comparison with translates in the proof of Theorem
\ref{thm:haarrado}. On the overlap $\Om'=\Om\cap(\Om-\tau)$ both $u$ and
$u_{\tau}=u(\cdot+\tau)$ minimize the same functional, and every point of
$\partial\Om'$ has either itself or its translate on $\partial\Om$, since
$\partial\Om'\subset\partial\Om\cup\partial(\Om-\tau)$; this is what turns the
boundary--to--interior slope $S$ into a global Lipschitz bound.
No convexity of $\Om$ is used.}
\label{fig:translates}
\end{figure}

Theorem \ref{thm:haarrado} is a Haar--Rado type statement (see
\cite{ACJ04,Cellina,MT} for the classical theory of Lipschitz regularity of
minimizers of functionals of the gradient) and it reduces the whole existence
question, in the regime $\Lam=\infty$, to a \emph{boundary} gradient estimate.
This is exactly the structure of the theory of \cite{BS1982} for the Minkowski
mean curvature operator, where the global gradient bound (\cite[Theorem 3.5]{BS1982}) is
obtained from a boundary estimate through the special structure of the
operator. Note that no convexity of $\Om$ is required; on the other hand, in a
non-convex domain the quantity $S$ may be much larger than $L$, since two points
which are close in $\R^{N}$ may be far apart within $\Om$, so that the criterion
$S<\sg$ becomes correspondingly restrictive. This is the same phenomenon as the
\emph{spacelike displacing} condition of \cite{BCP2021} recalled in
Section \ref{sec:applications}.

\subsection{The bounded slope condition}

Recall that $g$ satisfies the \emph{bounded slope condition} (BSC) of rank
$K>0$ on $\partial\Om$ if for every $y\in\partial\Om$ there exist affine
functions $\ell^{\pm}_{y}$ with
\[
\ell^{-}_{y}(y)=g(y)=\ell^{+}_{y}(y),\qquad
\ell^{-}_{y}\le g\le\ell^{+}_{y}\ \text{ on }\partial\Om,\qquad
|\nb\ell^{\pm}_{y}|\le K .
\]
Unless $g$ is affine, the BSC forces $\Om$ to be convex, and it is classically
satisfied with a
computable $K=K(\Om,g)$ whenever $\Om$ is uniformly convex with $C^{2}$
boundary and $g\in C^{1,1}(\partial\Om)$; see \cite{Cellina,MT}.

\begin{theorem}\label{thm:bsc}
Let $f\equiv0$ and let $g$ satisfy the BSC of rank $K<\sg$ on $\partial\Om$, so
that in particular $\Om$ is convex.
Let $u$ be the minimizer of $F$ on $S_{g}$ and, for $n\ge2$, let $u_{n}$ be
the minimizer of the truncated energy
\begin{equation}\label{eq:Fn}
F_{n}(v)=\sum_{p=2}^{n}\frac{a_{p}}{p}\int_{\Om}|\nb v|^{p}-\int_{\Om}fv
\end{equation}
among $v\in W^{1,n}(\Om)$ with $v-\bar g\in W^{1,n}_{0}(\Om)$. Then
\begin{equation}\label{eq:uniformK}
\|\nb u\|_{L^{\infty}(\Om)}\le K<\sg
\qquad\text{and}\qquad
\|\nb u_{n}\|_{L^{\infty}(\Om)}\le K<\sg\quad\text{for every }n\ge2 .
\end{equation}
In particular $u$ is a weak solution of \eqref{eq:main} and $u_{n}$ is a weak
solution of the truncated problem.
\end{theorem}

\begin{proof}
We give the proof for $u$; the one for $u_{n}$ is identical, with $G$ replaced
by $G_{n}(\xi)=\sum_{p\le n}\frac{a_{p}}{p}|\xi|^{p}$, which is convex, finite
and strictly convex on the whole of $\R^{N}$, so that Lemma
\ref{lem:comparison} applies verbatim; note that the rank $K$ does not depend
on $n$.

Fix $y\in\partial\Om$. The affine function $\ell^{+}_{y}$ satisfies
$|\nb\ell^{+}_{y}|\le K<\sg$ and is a weak solution of $\dive(\aaa(\nb v))=0$
(the field $\aaa(\nb\ell^{+}_{y})$ being constant), hence by Corollary
\ref{cor:VI2PDE} it is the minimizer of $F$ among the functions of
$S_{\ell^{+}_{y}|_{\partial\Om}}$. Since $u=g\le\ell^{+}_{y}$ on
$\partial\Om$, Lemma \ref{lem:comparison} (applied with $\lambda=0$, $w=u$ and
$u$ replaced by $\ell^{+}_{y}$) gives $u\le\ell^{+}_{y}$ in $\Om$; symmetrically
$u\ge\ell^{-}_{y}$. Therefore, for every $x\in\Om$,
\[
|u(x)-g(y)|=|u(x)-\ell^{\pm}_{y}(y)|\le K|x-y| ,
\]
so that $S\le K$ where $S$ is as in  \eqref{eq:S}, and Theorem \ref{thm:haarrado}
gives $\Lip(u,\Om)\le K$.
\end{proof}

\begin{remark}
Theorem \ref{thm:bsc} is the first instance in which we produce solutions of
\eqref{eq:main}, and it does so in all three regimes \eqref{eq:trichotomy}
simultaneously{, consistently with the fact that, when $f=0$, the}
obstruction \eqref{eq:giustiintro} of the next section is empty. Its real
strength, however, is the uniformity in $n$ of \eqref{eq:uniformK}, which is
what makes the truncation scheme converge; see Theorem \ref{thm:trunc} below.
For $f\ne0$ the comparison argument of Theorem
\ref{thm:haarrado} still applies whenever $f$ is constant, and only the
boundary estimate $S<\sg$ is missing: see Open Problem \ref{op:lip}.
\end{remark}

\begin{remark} \label{rem:hilberthaar}
Theorem \ref{thm:haarrado} and Theorem \ref{thm:bsc} belong to the classical
Hilbert--Haar theory of Lipschitz regularity for minimizers of functionals of
the gradient under the bounded slope condition, going back to Hartman and
Stampacchia and developed, without any growth or smoothness assumption on the
integrand, in \cite{Cellina,MT,Clarke}. In particular \cite{MT} treats
functionals of the form $\int_{\Om}[F(\nb v)+g(v)]$, with an \emph{autonomous}
lower order term, under the sole assumption of strict convexity, and derives a
Lipschitz regularity result for constrained minima. Since a constant datum
corresponds to $g(v)=-\lambda v$, this should provide the boundary estimate
$S<\sg$ missing in Theorem \ref{thm:haarrado} and thus extend Theorem
\ref{thm:bsc} to constant data; we have not carried out the details, and we
stress that the case of a datum $f=f(x)$ genuinely depending on $x$ is not
covered by that theory, nor by the comparison with translates of Theorem
\ref{thm:haarrado}, which requires translation invariance.
\end{remark}

\subsection{Convergence of the truncated problems}

Throughout this subsection we take $n>N$, which is no restriction since we are
interested in the limit $n\to\infty$; then $W^{1,n}(\Om)\hookrightarrow
C(\overline\Om)$, so that the functional $F_{n}$ of \eqref{eq:Fn} is well
defined and finite on $\bar g+W^{1,n}_{0}(\Om)$ for every $f\in L^{1}(\Om)$.

\begin{theorem}\label{thm:trunc}
Assume that $f\in L^{1}(\Om)$ and that Assumption \ref{ass:A} holds. Let $u$ be the
minimizer of $F$ on $S_{g}$ and let $u_{n}$ be the minimizer of $F_{n}$ as in
\eqref{eq:Fn}. Assume the uniform bound
\begin{equation}\label{eq:unifbound}
\|\nb u_{n}\|_{L^{\infty}(\Om)}\le K<\sg\qquad\text{for every }n\, { >N}.
\end{equation}
Then $u_{n}\to u$ uniformly on $\overline\Om$, $\|\nb u\|_{\infty}\le K$, $u$
is a weak solution of \eqref{eq:main}, and the convergence is geometric:
for every $\theta\in(K/\sg,1)$ there is $C>0$ such that
\begin{equation}\label{eq:rate}
\frac{a_{2}}{2}\int_{\Om}|\nb u_{n}-\nb u|^{2}\ \le\ F(u_{n})-F(u)\ \le\
|\Om|\sum_{p>n}\frac{a_{p}}{p}K^{p}\ \le\ C\theta^{\,n} .
\end{equation}
\end{theorem}

\begin{proof}
By \eqref{eq:unifbound} we have $u_{n}\in S_{g}$ for every $n$, so that
$F_{n}(u_{n})\le F_{n}(v)\le F(v)$ for every $v\in S_{g}$. As in the proof of
Theorem \ref{thm:min}, a subsequence of $u_{n}$ converges uniformly to some
$u^{*}\in S_{g}$ with $\|\nb u^{*}\|_{\infty}\le K$. For fixed $m\le n$ we
have $F_{m}(u_{n})\le F_{n}(u_{n})$, hence by lower semicontinuity
$F_{m}(u^{*})\le\liminf_{n}F_{n}(u_{n})\le F(v)$, and letting $m\to\infty$ by
monotone convergence $F(u^{*})\le F(v)$ for every $v\in S_{g}$. By uniqueness
$u^{*}=u$, and the whole sequence converges. Since $\|\nb u\|_{\infty}\le
K<\sg$, Corollary \ref{cor:VI2PDE} applies.

For the rate, using $F_{n}(u_{n})\le F_{n}(u)$,
\[
\begin{array}{l}
\displaystyle  F(u_{n})-F(u)=\big[F(u_{n})-F_{n}(u_{n})\big]+\big[F_{n}(u_{n})-F_{n}(u)\big]
-\big[F(u)-F_{n}(u)\big] \\\\ \displaystyle   \le\sum_{p>n}\frac{a_{p}}{p}\int_{\Om}|\nb u_{n}|^{p}
\le|\Om|\sum_{p>n}\frac{a_{p}}{p}K^{p} .\end{array}
\]
By \eqref{eq:sigma}, for every $\theta\in(K/\sg,1)$ one has $a_{p}K^{p}\le
C\theta^{p}$ for $p$ large, whence the last inequality in \eqref{eq:rate}.
Finally, $\xi\mapsto G(\xi)-\frac{a_{2}}{2}|\xi|^{2}$ is still convex, so that
$G$ is $a_{2}$-strongly convex and, $u$ being the minimizer,
\[\begin{array}{l}
 \displaystyle  F(u_{n})\ \ge\ F(u)+\int_{\Om}\big[\aaa(\nb u)\cdot\nb(u_{n}-u)-f(u_{n}-u)\big]
+\frac{a_{2}}{2}\int_{\Om}|\nb u_{n}-\nb u|^{2}
\  \\ \\  \displaystyle \ge\ F(u)+\frac{a_{2}}{2}\int_{\Om}|\nb u_{n}-\nb u|^{2},\end{array}
\]
where the bracket is nonnegative by Theorem \ref{thm:VI}.
\end{proof}

\begin{remark}\label{rem:rate}
Estimate \eqref{eq:rate} says that the truncation of the series is a
\emph{geometrically convergent} scheme, with a rate governed by the distance
of the solution from the constraint: the closer $K$ to $\sg$,
the slower the convergence, and the rate degenerates exactly at the threshold
identified in Section \ref{sec:nonexistence}.
\end{remark}

\section{Nonexistence for large data when
\texorpdfstring{$\Lam<\infty$}{Lambda<infinity}}\label{sec:nonexistence}

We come to the main obstruction. The mechanism is transparent: if $\Lam<\infty$
the flux field $\aaa(\nb u)$ is \emph{bounded} on the whole admissible set, and
a bounded flux can only support a limited amount of source.

\begin{theorem}\label{thm:nonexistence}
Assume $\Lam<\infty$ and let $u$ be a weak solution of \eqref{eq:main} in the
sense of Definition \ref{def:weak}, with $f\in L^{1}(\Om)$. Then
\begin{equation}\label{eq:giusti}
\Big|\int_{E}f\Big|\ \le\ \Lam\,P(E)
\qquad\text{for every set of finite perimeter } E\Subset\Om .
\end{equation}
\end{theorem}

\begin{proof}
Since $u\in S_{g}$ we have $|\nb u|\le\sg$ a.e., hence
\[
|\aaa(\nb u)|=\varphi(|\nb u|)\le\varphi(\sg^{-})=\Lam\qquad\text{a.e. in }\Om .
\]
Therefore, by \eqref{eq:weak}, for every $\psi\in C^{\infty}_{c}(\Om)$,
\begin{equation}\label{eq:fluxineq}
\Big|\int_{\Om}f\psi\Big|=\Big|\int_{\Om}\aaa(\nb u)\cdot\nb\psi\Big|
\le\Lam\int_{\Om}|\nb\psi| .
\end{equation}
Let now $E\Subset\Om$ have finite perimeter and let
$\psi_{\eps}=\chi_{E}*\rho_{\eps}$ with $\rho_{\eps}$ a standard mollifier and
$\eps<\dist(E,\partial\Om)$, so that $\psi_{\eps}\in C^{\infty}_{c}(\Om)$ and
$0\le\psi_{\eps}\le1$. Then $\psi_{\eps}\to\chi_{E}$ in $L^{1}$ and a.e., so
that $\int f\psi_{\eps}\to\int_{E}f$ by dominated convergence, while
$\int_{\Om}|\nb\psi_{\eps}|\to P(E)$ by the very definition of the perimeter.
Passing to the limit in \eqref{eq:fluxineq} gives \eqref{eq:giusti}.
\end{proof}

\begin{corollary}\label{cor:cheeger}
Assume $\Lam<\infty$.
\begin{enumerate}
\item[(i)] If $f\equiv\lambda$ with $|\lambda|>\Lam\,h(\Om)$, then
\eqref{eq:main} has no weak solution, whatever the boundary datum $g$.
\item[(ii)] More generally, given $0\le f\in L^{1}(\Om)$, $f\not\equiv0$, set
\begin{equation}\label{eq:Lambdaf}
\Lam_{f}:=\Lam\cdot\inf\Big\{\frac{\int_{\Om}|\nb\psi|}{\int_{\Om}f\psi}\ :\
\psi\in C^{\infty}_{c}(\Om),\ \psi\ge0,\ \int_{\Om}f\psi>0\Big\} .
\end{equation}
Then the problem $-\sum_{p}a_{p}\Delta_{p}u=\mu f$ in $\Om$, $u=g$ on
$\partial\Om$, admits no weak solution for $\mu>\Lam_{f}$.
\item[(iii)] In either case the minimizer $u$ of $F$ given by Theorem
\ref{thm:min} exists but satisfies $\|\nb u\|_{L^{\infty}(\Om)}=\sg$: the
gradient constraint is active.
\end{enumerate}
\end{corollary}

\begin{proof}
(i) and (ii) are immediate from \eqref{eq:fluxineq} and the definitions of
$h(\Om)$ and of $\Lam_{f}$; (iii) can be deduced by contradiction using  Corollary
\ref{cor:VI2PDE}.
\end{proof}

\begin{remark}{
The implication in Corollary \ref{cor:cheeger}(iii) cannot be reversed:
saturation of the gradient and nonexistence are two distinct events. Indeed on
$\Sigma_{u}$ the flux $\aaa(\nb u)$ is still well defined, of modulus $\Lam$, so
that the weak formulation \eqref{eq:weak} retains a meaning there; what fails,
in general, is only the possibility of choosing the Lagrange multiplier of the
constraint equal to zero, see Remark \ref{rem:cellina}. Theorem
\ref{thm:radial}(i) below exhibits, at the critical value $\lambda=\Lam h(B_{R})$, a
weak solution whose gradient attains $\sg$.}
\end{remark}

\begin{remark}\label{rem:comparisons}
Three comments are in order.

\emph{(a) The $1$-Laplacian replaces the $2$-Laplacian.} In \cite[Theorem
2.6]{P2010} the threshold for the zeroth order problem is expressed through
the first eigenvalue $\lambda_{1}(A,f)$ of the weighted $2$-Laplacian. Here
the same proof, testing the equation against a suitable positive function
and using the bound on the nonlinearity, produces the quantity
\eqref{eq:Lambdaf}, which is $\Lam$ times the first eigenvalue of the
$1$-Laplacian relative to the weight $f$ (\cite{KaSch,BSS2026}), and reduces to $\Lam h(\Om)$ for
$f\equiv1$ by the coarea formula. This is the precise sense in which the
present problem is the ``gradient'' counterpart of \cite{P2010}.

\emph{(b) Giusti's condition.} Inequality \eqref{eq:giusti} is formally
identical to the classical necessary condition
$|\int_{E}H|\le P(E)$ for the solvability of the prescribed mean curvature
equation $\dive(\nb u/\sqrt{1+|\nb u|^{2}})=H$ in Euclidean space, where the
flux is bounded by $1$ \cite{Giusti78,Giusti,MassariMiranda}. 

 The same argument, which only uses the flux bound, gives it as the case $\Lambda = 1$ (the Euclidean operator is outside our class, but see Remark \ref{rem:euclcheck} below)

\emph{(c) Gauss' law.} In the electrostatic case 
one has $\mathbf E=-\nb u$, $\mathbf D=\aaa(\nb u)$ and $f=\rho$, and
\eqref{eq:giusti} is nothing but $\int_{E}\rho=\oint_{\partial E}\mathbf
D\cdot\nu\le\|\mathbf D\|_{\infty}P(E)$: a constitutive law with a
\emph{saturating displacement} can only support charge distributions obeying
an isoperimetric bound (see Section \ref{sec:applications} below).
\end{remark}

\begin{remark}[Fenchel--Legendre duality]\label{rem:legendre}
The energy densities of the two mean curvature operators,
$\Phi_{E}(t)=\sqrt{1+t^{2}}-1$ and $\Phi_{M}(t)=1-\sqrt{1-t^{2}}$, are
Fenchel--Legendre conjugate. As is well known, this is equivalent to  the fact that $\varphi_{M}=\varphi_{E}^{-1}$. Conjugation therefore exchanges the
two structural constants (the saturation constant $\sigma$ and the total flux constant $\Lambda$) of the pair,
\[
(\sg,\Lam)(\Phi_{E})=(\infty,1),\qquad (\sg,\Lam)(\Phi_{M})=(1,\infty),
\]
that is, it exchanges a constraint on the gradient with a constraint on the flux,
and this is precisely why exactly one of the two operators carries the
obstruction of Theorem \ref{thm:nonexistence}.

We stress that this exchange is a feature of the pair above and not a general
rule. If $\sg(\Phi)<\infty$ and $\Phi(\sg)<\infty$, then
$\Phi^{*}(y)=\sup_{t\le\sg}(ty-\Phi(t))$ is finite for every $y$, so that
$\sg(\Phi^{*})=+\infty$ whatever the value of $\Lam(\Phi)$; for the logarithmic
operator of Example \ref{ex:log}, for instance, $\Phi^{*}$ is finite on the whole
of $\R$ and affine of slope $\sg=1$ for $y\ge\Lam=2$. The involution
$(\sg,\Lam)\mapsto(\Lam,\sg)$ holds between the family $\{\sg<\infty,\
\Lam=\infty\}$, which contains the regime (H2), and the family
$\{\sg=\infty,\ \Lam<\infty\}$, which contains the Euclidean operator, and it is
in this sense that the Minkowski and the Euclidean mean curvature operators are
dual to one another. The exchange is displayed in Figure \ref{fig:legendre}.
\end{remark}

\begin{figure}[ht]
\centering
\begin{tikzpicture}
\begin{axis}[width=.46\textwidth,height=.31\textwidth,
  xmin=0,xmax=2.6,ymin=0,ymax=1.75,
  xtick={1},xticklabels={$\sg=1$},ytick=\empty,
  xlabel={$t$},title={\normalfont$\Phi_{M}(t)=1-\sqrt{1-t^{2}}$}]
\fill[Shade] (axis cs:1,0) rectangle (axis cs:2.6,1.75);
\draw[gray!55,dashed] (axis cs:1,0)--(axis cs:1,1.75);
\addplot[RoyalBlue,domain=0:0.99999,samples=250,smooth] {1-sqrt(1-x*x)};
\addplot[RoyalBlue,only marks,mark=*,mark size=1.1] coordinates {(1,1)};
\node[RoyalBlue,fig label,anchor=west,align=left] at (axis cs:1.08,1.0)
  {$\Lam=+\infty$:\\ vertical tangent};
\end{axis}
\end{tikzpicture}\hspace{\figgap}
\begin{tikzpicture}
\begin{axis}[width=.46\textwidth,height=.31\textwidth,
  xmin=0,xmax=2.6,ymin=0,ymax=1.75,
  xtick={1},xticklabels={$1$},ytick=\empty,
  xlabel={$y$},title={\normalfont$\Phi_{E}(y)=\sqrt{1+y^{2}}-1$}]
\draw[gray!55,dashed] (axis cs:0,0)--(axis cs:2.6,1.6);
\addplot[BrickRed,domain=0:2.6,samples=200,smooth] {sqrt(1+x*x)-1};
\node[gray!70,fig label,anchor=north west] at (axis cs:1.62,0.62)
  {slope $\Lam=1$};
\node[BrickRed,fig label,anchor=north west,align=left] at (axis cs:0.06,1.70)
  {$\sg=+\infty$:\\ no constraint};
\end{axis}
\end{tikzpicture}
\caption{Fenchel--Legendre duality between the two mean curvature energies,
$\Phi_{E}^{*}=\Phi_{M}$. The energy of the Minkowski operator is finite on a
bounded interval and reaches its endpoint with a vertical tangent; its
conjugate, the energy of the Euclidean operator, is finite everywhere and has
the finite recession slope $1$. Conjugation exchanges the endpoint of the
domain with the recession slope, that is $(\sg,\Lam)=(1,\infty)$ with
$(\sg,\Lam)=(\infty,1)$: a constraint on the gradient becomes a constraint on
the flux, and only the latter obstructs solvability.}
\label{fig:legendre}
\end{figure}
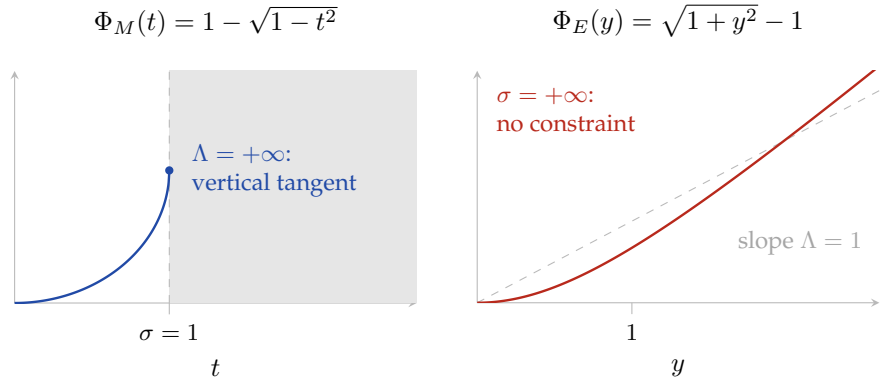

\begin{example}[the logarithmic operator: the model case of (H3)]\label{ex:log}
The model family \eqref{eq:model} never realizes the regime (H3), since
$\varphi(t)=t(1-t^{q})^{-\alpha}\to+\infty$ as $t\to1^{-}$ for every
$\alpha>0$. A model case for (H3), which is the exact counterpart of
\cite[Example 2]{P2010}, is obtained by taking
\[
a_{2}=1,\qquad a_{p}=\frac{1}{(p-1)(p-2)}\quad (p\ge3) .
\]
Then $\sg=1$ and, using $\sum_{k\ge2}\frac{t^{k}}{k(k-1)}=t+(1-t)\log(1-t)$,
\begin{equation}\label{eq:logmodel}\begin{array}{l} \displaystyle 
\varphi(t)=2t+(1-t)\log(1-t),\qquad
\mathcal A(t) =2+\frac{(1-t)\log(1-t)}{t},\qquad \\ \\  \displaystyle 
\Phi(t)=t^{2}-\frac{1}{4}+\frac{(1-t)^{2}}{4}-\frac{(1-t)^{2}}{2}\log(1-t),\end{array}
\end{equation}
so that
\[
\Lam=\varphi(1^{-})=2<\infty,\qquad \Phi(1)=\tfrac34<\infty .
\]
The corresponding equation,
\[
-\dive\Big(\Big[2+\frac{(1-|\nb u|)\log(1-|\nb u|)}{|\nb u|}\Big]\nb u\Big)=f,
\qquad |\nb u|\le1,
\]
is uniformly elliptic on the admissible set, $1\le\mathcal A\le2$,   and
nevertheless, by Corollary \ref{cor:cheeger}, it has no weak solution as soon
as $f\equiv\lambda>2h(\Om)$. The obstruction is not due to a degeneracy of the
operator but to the finiteness of the total flux.
\end{example}

\subsection{The saturation region}

When $\Lam<\infty$ and the datum is large, the minimizer survives but the
equation does not, and the difference between the two is concentrated on
\begin{equation}\label{eq:satset}
\Sigma_{u}:=\{x\in\Om:\ |\nb u(x)|=\sg\},
\end{equation}
which we call the \emph{saturation region}. On $\Om\setminus\Sigma_{u}$ the
variational inequality \eqref{eq:VI} is an equation, while on $\Sigma_{u}$ the
constraint carries a Lagrange multiplier. This is the exact analogue of the
flat zones $\{v=\sg\}$ of \cite[\S5]{P2010}, and it is a familiar object in
several applied contexts: the \emph{light rays} along which a maximal
spacelike hypersurface becomes lightlike \cite{BS1982}, the \emph{plastic
regions} of elastic--plastic torsion, where the gradient constrained problem
is classically reformulated as an obstacle problem with obstacle
$\sg\dist(x,\partial\Om)$ \cite{BrezisSibony}, and the saturation zones of
Born--Infeld electrostatics near concentrated charges. In Section \ref{sec:radial}
we compute $\Sigma_{u}$ explicitly in the radial case.

\section{The radial problem: explicit solution and sharpness of the
threshold}\label{sec:radial}

Let $\Om=B_{R}\subset\R^{N}$, $g\equiv0$ and let $f=f(|x|)\ge0$ be radial and
bounded. Set
\begin{equation}\label{eq:calF}
\mathcal F(s):=\frac{1}{s^{N-1}}\int_{0}^{s}f(r)\,r^{N-1}\,dr ,
\qquad s\in(0,R),
\end{equation}
which for $f\equiv\lambda$ reduces to $\mathcal F(s)=\lambda s/N$.

\begin{theorem}\label{thm:radial}
Under the above assumptions the minimizer $u$ of $F$ on $S_{0}$ is radial,
$u(x)=U(|x|)$, and is given explicitly by
\begin{equation}\label{eq:radialsol}
U(r)=\int_{r}^{R}\varphi^{-1}\big(\min\{\mathcal F(s),\Lam\}\big)\,ds ,
\qquad 0\le r\le R,
\end{equation}
with the convention $\varphi^{-1}(\Lam)=\sg$. Consequently:
\begin{enumerate}
\item[(i)] if $\mathcal F\le\Lam$ a.e.\ in $(0,R)$, then $u$ is the unique
weak solution of \eqref{eq:main} and
$\|\nb u\|_{\infty}=\varphi^{-1}\big(\sup\mathcal F\big)$, which is $<\sg$ if and
only if $\sup_{(0,R)}\mathcal F<\Lam$;
\item[(ii)] if $\mathcal F(s)>\Lam$ on a set of positive measure, then
\eqref{eq:main} has no weak solution and the saturation region is
\[
\Sigma_{u}=\{x:\ \mathcal F(|x|)\ge\Lam\},
\]
a spherical shell of positive measure;
\item[(iii)] for $f\equiv\lambda>0$ a weak solution exists if and only if
\[
\lambda\le\Lam\,h(B_{R})=\frac{N\Lam}{R},
\]
and for $\lambda>N\Lam/R$ the saturation region is the shell
$\{N\Lam/\lambda<|x|<R\}$. In particular the threshold of Corollary
\ref{cor:cheeger} is sharp.
\end{enumerate}
\end{theorem}

\begin{proof}
Since $F$ and $S_{0}$ are invariant under rotations and the minimizer is
unique, $u$ is radial. Write $u(x)=U(|x|)$ with $U(R)=0$ and set
$v(s)=-U'(s)$, so that $U(r)=\int_{r}^{R}v$ and, by Fubini,
\[\begin{array}{l} \displaystyle 
\int_{B_{R}}fu=\omega_{N-1}\int_{0}^{R}f(r)r^{N-1}\!\!\int_{r}^{R}\!v(s)\,ds\,dr
\\ \\  \displaystyle =\omega_{N-1}\int_{0}^{R}v(s)\Big(\int_{0}^{s}\!f r^{N-1}dr\Big)ds
=\omega_{N-1}\int_{0}^{R}s^{N-1}\mathcal F(s)v(s)\,ds ,\end{array}
\]
$\omega_{N-1}$ denoting the surface measure of the unit sphere. Therefore
\[
F(u)=\omega_{N-1}\int_{0}^{R}s^{N-1}\Big[\Phi(v(s))-\mathcal F(s)v(s)\Big]ds ,
\]
and the minimization \emph{decouples pointwise} in $s$ over the set
$\{0\le v\le\sg\}$ (that $v\ge0$ at the minimum is clear since $\mathcal
F\ge0$). The function $v\mapsto\Phi(v)-\mathcal F(s)v$ is strictly convex with
derivative $\varphi(v)-\mathcal F(s)$; hence its minimum over $[0,\sg]$ is
attained at $v=\varphi^{-1}(\mathcal F(s))$ if $\mathcal F(s)<\Lam$, and at
$v=\sg$ if $\mathcal F(s)\ge\Lam$, which is \eqref{eq:radialsol}.

 {For (i), the field $\aaa(\nb u)$ has modulus $\varphi(|\nb u|)=\mathcal
F(|x|)\le\Lam$ and satisfies $-\dive\aaa(\nb u)=f$ in the weak sense by
construction, so that $u$ is a weak solution; it is the unique one because any
weak solution is the minimizer, by the converse part of Corollary
\ref{cor:VI2PDE}. Note that the strict bound $\|\nb u\|_{\infty}<\sg$ is
\emph{not} needed here: the constraint may be attained, provided the flux stays
below $\Lam$.} 

For (ii), no weak solution can exist: by Theorem \ref{thm:nonexistence}
applied with $E=B_{s}$, since $\int_{B_{s}}f=\omega_{N-1}s^{N-1}\mathcal F(s)$
and $P(B_{s})=\omega_{N-1}s^{N-1}$, condition \eqref{eq:giusti} reads exactly
$\mathcal F(s)\le\Lam$, which fails for $s$ in a set of positive measure. The
minimizer \eqref{eq:radialsol} then satisfies $|\nb u|=\sg$ precisely on
$\{\mathcal F\ge\Lam\}$. 

Finally, (iii) is the case $\mathcal F(s)=\lambda s/N$,
together with $h(B_{R})=N/R$.
\end{proof}

\begin{remark}
The proof shows that in the radial case the necessary condition
\eqref{eq:giusti}, tested only on concentric balls, is also \emph{sufficient}.
It also exhibits, for every $N$ and every operator in the
regime (H3), an explicit minimizer with a saturation region of positive
measure, displayed in Figure \ref{fig:radialsat}.
\end{remark}

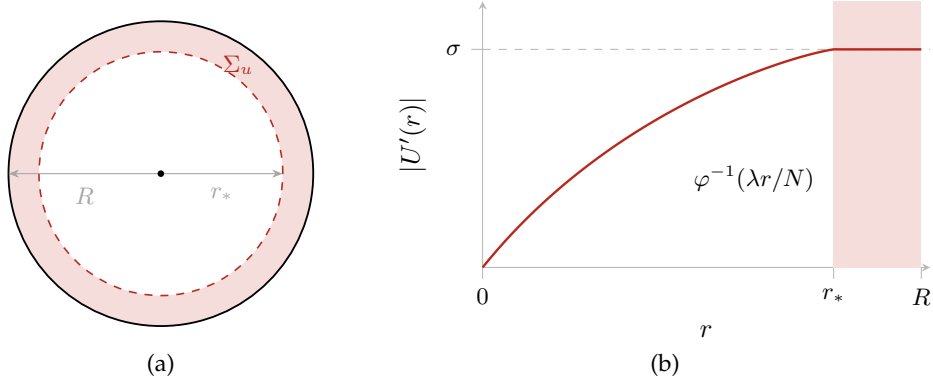
\begin{figure}[ht]
\centering
\begin{tabular}{@{}c@{\hspace{2.6em}}c@{}}
\begin{tikzpicture}[scale=1.55,baseline=(current bounding box.center)]
\begin{scope}
  \clip (0,0) circle (1.3);
  \fill[BrickRed!16] (0,0) circle (1.3);
  \fill[white] (0,0) circle (1.04);
\end{scope}
\draw[black,line width=.7pt] (0,0) circle (1.3);
\draw[BrickRed,line width=.6pt,dashed] (0,0) circle (1.04);
\draw[-{Stealth[length=4pt]},gray!70,line width=.4pt] (0,0)--(1.04,0);
\node[gray!70,fig label,anchor=north] at (0.52,-0.03) {$r_{*}$};
\draw[-{Stealth[length=4pt]},gray!70,line width=.4pt] (0,0)--(-1.3,0);
\node[gray!70,fig label,anchor=north] at (-0.65,-0.03) {$R$};
\node[BrickRed,fig label] at (0.66,0.93) {$\Sigma_{u}$};
\fill (0,0) circle (.8pt);
\end{tikzpicture}
&
\begin{tikzpicture}[baseline=(current bounding box.center)]
\begin{axis}[width=.50\textwidth,height=.34\textwidth,
  xmin=0,xmax=1.02,ymin=0,ymax=1.22,
  xtick={0,0.8,1},xticklabels={$0$,$r_{*}$,$R$},
  ytick={1},yticklabels={$\sg$},
  xlabel={$r$},ylabel={$|U'(r)|$}]
\fill[BrickRed!16] (axis cs:0.8,0) rectangle (axis cs:1,1.22);
\draw[gray!55,dashed] (axis cs:0,1)--(axis cs:1.02,1);
\addplot[BrickRed,smooth] coordinates {(0.0000,0.0000) (0.0125,0.0308) (0.0250,0.0606) (0.0375,0.0896) (0.0500,0.1178) (0.0625,0.1452) (0.0750,0.1718) (0.0875,0.1978) (0.1000,0.2230) (0.1125,0.2477) (0.1250,0.2717) (0.1375,0.2951) (0.1500,0.3180) (0.1625,0.3403) (0.1750,0.3622) (0.1875,0.3835) (0.2000,0.4043) (0.2125,0.4246) (0.2250,0.4445) (0.2375,0.4640) (0.2500,0.4830) (0.2625,0.5017) (0.2750,0.5199) (0.2875,0.5377) (0.3000,0.5552) (0.3125,0.5723) (0.3250,0.5890) (0.3375,0.6053) (0.3500,0.6214) (0.3625,0.6370) (0.3750,0.6524) (0.3875,0.6674) (0.4000,0.6822) (0.4125,0.6966) (0.4250,0.7107) (0.4375,0.7245) (0.4500,0.7380) (0.4625,0.7512) (0.4750,0.7641) (0.4875,0.7768) (0.5000,0.7891) (0.5125,0.8012) (0.5250,0.8130) (0.5375,0.8246) (0.5500,0.8358) (0.5625,0.8468) (0.5750,0.8576) (0.5875,0.8680) (0.6000,0.8782) (0.6125,0.8881) (0.6250,0.8978) (0.6375,0.9072) (0.6500,0.9163) (0.6625,0.9251) (0.6750,0.9337) (0.6875,0.9420) (0.7000,0.9499) (0.7125,0.9576) (0.7250,0.9650) (0.7375,0.9720) (0.7500,0.9786) (0.7625,0.9849) (0.7750,0.9906) (0.7875,0.9958) (0.8000,1.0000) (0.8125,1.0000) (0.8250,1.0000) (0.8375,1.0000) (0.8500,1.0000) (0.8625,1.0000) (0.8750,1.0000) (0.8875,1.0000) (0.9000,1.0000) (0.9125,1.0000) (0.9250,1.0000) (0.9375,1.0000) (0.9500,1.0000) (0.9625,1.0000) (0.9750,1.0000) (0.9875,1.0000) (1.0000,1.0000)};
\node[fig label,anchor=south east] at (axis cs:0.78,0.30)
  {$\varphi^{-1}(\lambda r/N)$};
\end{axis}
\end{tikzpicture}
\\[4pt]
{\footnotesize (a)} & {\footnotesize (b)}
\end{tabular}
\caption{The saturation region of Theorem \ref{thm:radial} for the logarithmic
operator of Example \ref{ex:log} in $B_{1}\subset\R^{2}$: (a) the saturation
shell $\Sigma_{u}$; (b) the modulus of the gradient along a radius. Here $\sg=1$,
$\Lam=2$ and the threshold is $\Lam h(B_{1})=4$; we take $\lambda=5$, so that
the inner radius of the shell is $r_{*}=N\Lam/\lambda=0.8$. On $\{r<r_{*}\}$ the equation is satisfied and the
gradient is strictly below the constraint; on the shell $\Sigma_{u}$ the
gradient saturates, the flux is stuck at $\Lam$, and no weak solution of
\eqref{eq:main} exists.}
\label{fig:radialsat}
\end{figure}

\begin{remark}[Born--Infeld]
For the Minkowski operator \eqref{eq:mink}, $\varphi(t)=t/\sqrt{1-t^{2}}$ and
$\varphi^{-1}(y)=y/\sqrt{1+y^{2}}$, so that \eqref{eq:radialsol} gives, for
$f\equiv\lambda$, the classical explicit radial solution
\[
U(r)=\int_{r}^{R}\frac{\lambda s/N}{\sqrt{1+(\lambda s/N)^{2}}}\,ds
=\frac{N}{\lambda}\Big[\sqrt{1+\Big(\frac{\lambda R}{N}\Big)^{2}}
-\sqrt{1+\Big(\frac{\lambda r}{N}\Big)^{2}}\Big],
\]
which exists for every $\lambda>0$, in agreement with $\Lam=+\infty$; its
gradient tends to $1$ as $\lambda\to\infty$ but never reaches it.
\end{remark}

\begin{remark}[a check outside the class]\label{rem:euclcheck}
The proof of Theorem \ref{thm:radial} only uses that $\varphi$ is an increasing
homeomorphism of $[0,\sg)$ onto $[0,\Lam)$, and it therefore applies verbatim to
operators which are not series of $p$-Laplacians. Consider the \emph{Euclidean}
mean curvature operator: there $\varphi(y)=y/\sqrt{1+y^{2}}$, so that
$\sg=+\infty$ but $\Lam=\varphi(+\infty)=1$. Formula \eqref{eq:radialsol} then
predicts that the radial Dirichlet problem in $B_{R}$ with $f\equiv\lambda$ is
solvable if and only if
\[
\lambda<\Lam\,h(B_{R})=\frac{N}{R};
\]
note that here the inequality is strict, precisely because $\sg=+\infty$: at
$\lambda=N/R$ formula \eqref{eq:radialsol} would produce a function with
unbounded gradient, whereas for $\sg<\infty$ the endpoint is admissible and
Theorem \ref{thm:radial}(iii) holds with a non-strict inequality. For $R=1$ the
above is exactly the sharp condition $|H|<N$ obtained in
\cite[Theorem 2.2 and the subsequent remark]{BJM2009b} by a fixed point
argument, together with the explicit solution
$u(r)=\frac{N}{H}\big(\sqrt{1-H^{2}/N^{2}}-\sqrt{1-H^{2}r^{2}/N^{2}}\big)$.
This is an independent confirmation that the relevant quantity is the
saturation flux $\Lam$, here finite because the displacement saturates even
though the gradient is unconstrained, and not the energy, and it makes
precise the observation of \cite[Introduction]{Bayard} that euclidean
obstructions disappear in the lorentzian context: they disappear exactly when
$\Lam=+\infty$.
\end{remark}

\section{The one dimensional problem}\label{sec:1d}

In dimension one everything can be computed. Let $\Om=(d,e)$, $u(d)=a$,
$u(e)=b$, $f\in L^{1}(d,e)$ and let
\[
\hat\varphi(s)=\operatorname{sign}(s)\,\varphi(|s|)
\]
be the odd extension of the flux, an increasing homeomorphism of $(-\sg,\sg)$
onto $(-\Lam,\Lam)$. Set $\mathcal I(x)=\int_{d}^{x}f$.

\begin{proposition}\label{prop:1d}
A function $u$ is a weak solution of \eqref{eq:main} on $(d,e)$ if and only if
there exists $c\in\R$ with
\begin{equation}\label{eq:1dsol}
u'(x)=\hat\varphi^{-1}\big(c-\mathcal I(x)\big)\ \text{ a.e.},\qquad
\|c-\mathcal I\|_{L^{\infty}(d,e)}{\le}\Lam,\qquad
\Theta(c):=\int_{d}^{e}\hat\varphi^{-1}\big(c-\mathcal I\big)=b-a ,
\end{equation}
 {with the convention $\hat\varphi^{-1}(\pm\Lam)=\pm\sg$.}
The map $\Theta$ is continuous and strictly increasing on the interval
$I:={\big[}\max\mathcal I-\Lam,\ \min\mathcal I+\Lam{\big]}$, which is
nonempty if and only if
\begin{equation}\label{eq:osc}
\operatorname{osc}_{[d,e]}\mathcal I{\le}2\Lam .
\end{equation}
Consequently the problem is solvable if and only if \eqref{eq:osc} holds and
$b-a\in\Theta(I)$, and the solution is then unique.
\end{proposition}

\begin{proof}
Integrating the equation once gives $\hat\varphi(u')=c-\mathcal I$ for some
constant $c$; the constraint $|u'|{\le}\sg$ is equivalent to
$|c-\mathcal I|{\le}\Lam$, and the boundary conditions to $\Theta(c)=b-a$.
Strict monotonicity and continuity of $\Theta$ follow from those of
$\hat\varphi^{-1}$ and dominated convergence.
\end{proof}

Note that \eqref{eq:osc} is precisely condition \eqref{eq:giusti} tested on
intervals $E=(x_{1},x_{2})\Subset(d,e)$, for which $P(E)=2$: {in dimension
one the necessary condition of Theorem \ref{thm:nonexistence} is therefore also
sufficient for homogeneous boundary data}.

\begin{corollary}\label{cor:1dLam}
If $\Lam=+\infty$ (regimes (H1) and (H2)) then $I=\R$,
$\Theta(\R)=(-\sg(e-d),\sg(e-d))$, and the one dimensional problem is solvable
for \emph{every} $f\in L^{1}(d,e)$ if and only if
\[
L=\frac{|b-a|}{e-d}<\sg .
\]
If $L=\sg$ the only element of $S_{g}$ is the affine function $u(x)=a+L(x-d)$,
which solves \eqref{eq:main} if and only if $f\equiv0$.
\end{corollary}

\begin{remark}{
Note that the affine function does solve the \emph{series} equation with $f=0$
even though the non-expanded operator is singular there, since all the terms
$\Delta_{p}u$ vanish separately.}
\end{remark}

\begin{example}[$\mathcal A(t)=(1-t)^{-1}$, i.e.\ $a_{p}\equiv1$]\label{ex:1d}
Here $\sg=1$, $\Phi(t)=-\log(1-t)-t$, $\Phi(1)=\Lam=+\infty$: regime (H1).
One has $\hat\varphi^{-1}(y)=y/(1+|y|)$ and, for $f\equiv k>0$,
\eqref{eq:1dsol} gives
\[
u'(x)=\frac{c-kx}{1+|c-kx|} ,
\]
which integrates to the piecewise logarithmic profiles
\[
u(x)=x-d+a+\frac{1}{k}\log\big(\gamma(d-x)+1\big),\qquad
u(x)=e-x+b+\frac{1}{k}\log\big(\alpha(x-e)+1\big),
\]
glued at the maximum point $x_{*}=c/k$ when $x_{*}\in(d,e)$, with
$\alpha=k/(1+k(e-x_{*}))$ and $\gamma=k/(1+k(x_{*}-d))$; the solution is
$C^{2}$ but not $C^{3}$ at $x_{*}$. By Corollary \ref{cor:1dLam} a solution
exists for every $k$ as long as $|a-b|<e-d$, and the monotone regime
$|a-b|\ge e-d-\frac1k\log(1+k(e-d))$ is the one in which $x_{*}\notin(d,e)$.
\end{example}

\begin{example}[the logarithmic operator of Example \ref{ex:log}]\label{ex:1dlog}
Here $\sg=1$, $\Lam=2$ and $h((0,1))=2$. For $\Om=(0,1)$, $a=b=0$ and
$f\equiv\lambda$ one has $\mathcal I(x)=\lambda x$ and, by symmetry,
$c=\lambda/2$; condition \eqref{eq:osc} reads $\lambda\le4=\Lam h(\Om)$, in
agreement with Corollary \ref{cor:cheeger}. For $\lambda>4$ the minimizer is
\[
u'(x)=\hat\varphi^{-1}\Big(T_{\Lam}\big(\lambda(\tfrac12-x)\big)\Big),
\]
$T_{\Lam}$ being the truncation at level $\Lam$, and its saturation region is
$\{|x-\frac12|>\Lam/\lambda\}$, a pair of intervals on which $u$ is affine with
slope $\pm1$.
\end{example}

\section{Relevance to physical models}\label{sec:applications}

\subsection{Born--Infeld electrostatics and the maximal admissible charge}

In the electrostatic regime of a nonlinear electrodynamics with Lagrangian
density $\mathcal L=\mathcal L(|\mathbf E|)$, the potential $u$, the electric
field $\mathbf E=-\nb u$ and the electric displacement
$\mathbf D=\mathcal L'(|\mathbf E|)\frac{\mathbf E}{|\mathbf E|}$ are related by
Gauss' law $\dive\mathbf D=\rho$. Writing $\mathbf D=\mathcal A(|\mathbf
E|)\mathbf E$ we are exactly in the setting of \eqref{eq:main} with $f=\rho$,
and the two structural constants acquire a physical meaning:
\[
\sg=\text{maximal field strength},\qquad
\Lam=\text{maximal displacement}.
\]
The Born--Infeld theory \cite{BI1934}, $\mathcal
L=b^{2}(1-\sqrt{1-|\mathbf E|^{2}/b^{2}})$, is the case $\sg=b$,
$\Lam=+\infty$: the field is bounded, the displacement is not, and this is
precisely why Born--Infeld electrostatics accommodates arbitrary charge
distributions, including point charges \cite{BdAP}. Theorem
\ref{thm:nonexistence} shows what happens for constitutive laws with
saturating displacement:

\begin{corollary}\label{cor:charge}
Let the constitutive law $\mathbf D=\mathcal A(|\mathbf E|)\mathbf E$ be a
series of $p$-Laplacian type, with maximal displacement $\Lam<\infty$. Then, a
charge density $\rho$ admits an electrostatic potential with $|\mathbf
E|\le\sg$ in $\Om$ only if
\[
\Big|\int_{E}\rho\Big|\le\Lam\,P(E)\qquad\text{for every }E\Subset\Om .
\]
In particular a uniform density $\rho$ in a ball of radius $R$ is admissible
only for $\rho\le N\Lam/R$, i.e.\ the total charge enclosed by a surface
cannot exceed $\Lam$ times its area.
\end{corollary}

We are not aware of a statement of this kind in the literature on nonlinear
electrodynamics; it is the exact electrostatic counterpart of the
Massari--Miranda obstruction in the theory of surfaces of prescribed mean
curvature, and it identifies unboundedness of $\mathbf D${, rather than
boundedness of $\mathbf E$,} as the property responsible for the good
behavior of the Born--Infeld model.

\subsection{Convergence rate of the weak field expansion}

The truncations $F_{n}$ of \eqref{eq:Fn} are, in the electrostatic language,
the successive \emph{post-Maxwellian corrections} of the theory: $n=2$ is the
Poisson equation, $n=4$ adds the first correction $\dive(|\mathbf E|^{2}\mathbf
E)$, and so on. Theorems \ref{thm:bsc} and \ref{thm:trunc} give a convergence statement for this expansion at the
level of the boundary value problem, together with the quantitative rate
\[
   \|\nb u_{n}-\nb u\|_{L^{2}(\Om)}^{2}\le\frac{2|\Om|}{a_{2}}
   \sum_{p>n}\frac{a_{p}}{p}K^{p}\le C\theta^{n}\quad
   \text{for every }\theta\in(K/\sg,1),
\]
where $K=\sup_{n}\|\nb u_{n}\|_{\infty}$.  
The expansion converges geometrically as long as the field stays away from its
maximal value, and degenerates precisely when the field saturates{, which is
the regime where the perturbative treatment is expected to break down. We stress
that the statement is conditional upon the uniform bound \eqref{eq:unifbound},
which Theorem \ref{thm:bsc} provides for $f\equiv0$}.

\subsection{Maximal spacelike hypersurfaces}

Let us come back to  \eqref{eq:mink}; here  $u$ can be seen as  the graph function of a spacelike hypersurface of
$\R^{N,1}$ with prescribed mean curvature $f$, the constraint
$\|\nb u\|_{\infty}\le1$ is the spacelike condition, and the compatibility
requirement $L\le\sg=1$ on the boundary datum is {exactly the requirement that 
guarantees that $g$ admits a weakly spacelike extension. We recall the terminology of
\cite[\S1, p.~132]{BS1982}: the graph of $u\in C^{0,1}(\Om)$ is \emph{weakly
spacelike} if $|Du|\le1$ a.e.\ in $\Om$, \emph{spacelike} if
$|u(x)-u(y)|<|x-y|$ whenever $x,y\in\Om$ and the segment $\overline{xy}$ lies in
$\Om$, and \emph{strictly spacelike} if it is spacelike, $u\in C^{1}(\Om)$ and
$|Du|<1$ in $\Om$; in \cite{BS1982} the authors observe that weak spacelikeness is
enough for the area integral to be defined, whereas strict spacelikeness is
needed for the Lorentz mean curvature to be defined at all. This is, in
geometric language, precisely the distinction between our Theorem \ref{thm:min}
and our Corollary \ref{cor:VI2PDE}. When $\Om$ is not convex the natural
condition on $g$ involves the geodesic rather than the Euclidean distance,
compare the notion of \emph{spacelike displacing} datum of \cite{BCP2021}, so
that $L\le\sg$ as defined in \eqref{eq:L} is sufficient but not necessary}. Our
results give a variational
proof of the existence of a minimizer for every $f\in L^{1}$ and every weakly
spacelike datum (Theorem \ref{thm:min}), of the fact that the minimizer solves
the equation as soon as it is uniformly spacelike (Corollary
\ref{cor:VI2PDE}), and of the reduction of the latter to a boundary gradient
estimate (Theorem \ref{thm:haarrado}){, which is the structure of the proof
of \cite{BS1982}, here isolated from the specific geometry}. Since
$\Lam=+\infty$, Theorem \ref{thm:nonexistence} is empty, consistently with the
solvability of the Dirichlet problem for arbitrary bounded mean curvature.
Conversely, our framework shows that the obstruction reappears for any
Lorentzian-type operator with finite $\Lam$, and that the light rays of
\cite{BS1982} and the saturation regions of Section \ref{sec:nonexistence} stand for  the
same phenomenon.

\subsection{Gradient constrained models}

Finally, the family \eqref{eq:main} interpolates between the Laplacian
(only $a_{2}\neq0$) and the purely gradient constrained models of
elastic--plastic torsion and of sandpile growth
\cite{BrezisSibony,CaffarelliFriedman,EFG,Prigozhin}: when the mass of the
sequence $\{a_{p}\}$ concentrates on large $p$, $\Phi$ converges to the
indicator function of $\{|\xi|\le\sg\}$ and $F$ to the energy of the
constrained model. Theorem \ref{thm:radial} is, from this point of view, a
regularized version of the classical explicit sandpile solution
$\sg\dist(x,\partial\Om)$, to which formally  converges in the said limit.

\subsection{Flux-limited diffusion}
The proof of Theorem \ref{thm:nonexistence} only uses the bound
$|\aaa(\nb u)|\le\Lam$ on the flux. Bounded fluxes are the defining feature of
flux-limited, or flux-saturated, diffusion, where the flux stays bounded as the
gradient becomes large; see \cite{CCCSS15} for a survey and
\cite{CMSV11,CCCSS16} for applications to morphogen transport and pattern
formation. In the regime (H3) the same saturation mechanism appears in a
stationary, gradient constrained setting, where it manifests itself as the upper
bound \eqref{eq:giusti} on the admissible sources rather than, as in the
evolutionary setting, through the formation of sharp fronts.

\section{Concluding remarks and future perspectives}\label{sec:open}

The picture that emerges from our results is governed by a single chain of implications, valid in
all three regimes \eqref{eq:trichotomy} and independent of any a priori estimate:
the constrained energy always has a unique minimizer (Theorem \ref{thm:min}),
the minimizer always solves the variational inequality \eqref{eq:VI} (Theorem
\ref{thm:VI}), and it solves the equation \eqref{eq:main} exactly when the
gradient constraint is inactive (Corollary \ref{cor:VI2PDE}). Whether the
constraint can be active, and whether an active constraint is compatible with
solvability, is decided by the saturation flux $\Lam$ and not by the energy
$\Phi(\sg)$: this is the main structural point of the paper, and it is what
distinguishes the present problem from its zeroth order ancestor \cite{P2010},
where the two quantities coincide. Table \ref{tab:summary} collects the
resulting classification.

\begin{table}[ht]
\centering
\small
\begin{tabular}{|l|c|c|l|l|}
\hline
 & $\Phi(\sg)$ & $\Lam$ & gradient constraint & solvability of \eqref{eq:main}\\
\hline
(H1) & $\infty$ & $\infty$ & inactive a.e., Rem.~\ref{rem:H1ae}
     & no obstruction\\
(H2) & $<\infty$ & $\infty$ & inactive a.e.\ if $L<\sg$, Cor.~\ref{cor:H2nosat}
     & no obstruction\\
(H3) & $<\infty$ & $<\infty$ & active for large data, Thm.~\ref{thm:radial}
     & $|\int_{E}f|\le\Lam P(E)$, Thm.~\ref{thm:nonexistence}\\
\hline
\end{tabular}
\medskip
\caption{The trichotomy \eqref{eq:trichotomy} and its consequences. In the
first two rows a solution exists as soon as a Lipschitz bound below $\sg$ is
available, which we obtain under the bounded slope condition for $f\equiv 0$ (Theorem
\ref{thm:bsc}); in the third row solutions cease to exist for
$f\equiv\lambda>\Lam h(\Om)$, a threshold which is attained on balls (Theorem
\ref{thm:radial}) and in dimension one (Proposition \ref{prop:1d}). The model
cases are $a_{p}\equiv1$ for (H1), the Born--Infeld and Minkowski operator
\eqref{eq:mink} for (H2), and the logarithmic operator of Example \ref{ex:log}
for (H3).}
\label{tab:summary}
\end{table}

Two by-products seem to us of independent interest. On the side of nonlinear
electrostatics, the obstruction \eqref{eq:giusti} is Gauss' law combined with a
bound on the displacement field, and it produces an isoperimetric bound on the
admissible charge densities (Corollary \ref{cor:charge}); the borderline
character of the Born--Infeld model is thereby traced back to the unboundedness
of $\mathbf D$, rather than to the boundedness of $\mathbf E$. On the geometric
side, the Euclidean obstruction of Giusti and Massari--Miranda and its absence
in the Lorentzian setting appear as the two cases $\Lam<\infty$ and
$\Lam=\infty$ of one criterion, an equivalence made quantitative by Remark
\ref{rem:euclcheck}.

As future perspectives, the most natural continuation concerns the a priori Lipschitz bound in the
regime $\Lam=\infty$, where by Corollary \ref{cor:VI2PDE} it is the only
ingredient still missing for solvability. It is worth stressing how narrow the
gap is: by Remark \ref{rem:H1ae} and Corollary \ref{cor:H2nosat} the constraint
is already known to be inactive \emph{almost everywhere} whenever $L<\sg$, so
that what has to be gained is only the passage from an almost everywhere bound
to a uniform one.

\begin{openproblem}\label{op:lip}
Extend Theorem \ref{thm:bsc} to a nonzero datum, and in particular to a datum
$f=f(x)$ genuinely depending on $x$. For $f$ constant, Theorem
\ref{thm:haarrado} reduces the question to the boundary estimate $S<\sg$ for the
quantity \eqref{eq:S}, which should follow from the Hilbert--Haar theory recalled
in Remark \ref{rem:hilberthaar}. For nonconstant $f$ even the comparison with
translates is unavailable, since the functional is no longer translation
invariant, and a different route, for instance an adaptation of the iteration of
\cite[Theorem 3.5]{BS1982}, which requires structure conditions on $\mathcal A$
and on $\nb\big[\mathcal A(|\xi|)\big]$, seems necessary.
\end{openproblem}

Three further directions seem to us worth pursuing.

\smallskip

\emph{Regularity.} Since the $\Delta_{2}$ condition fails whenever $\sg<\infty$,
none of the available regularity theories for functionals with $(p,q)$ or
nonstandard growth \cite{M2006,M2020,ELM2004,CM2015,BCM2018,L1991} applies
directly, and the results of Section \ref{sec:existence} give Lipschitz continuity but
nothing beyond. For the Minkowski operator \eqref{eq:mink}, however, a bound of the form
$\|\nb u\|_{\infty}<\sg=1$, together with the continuity of the gradient, has
been very recently obtained in \cite{DFM2026} by methods of nonlinear potential
theory for nonuniformly elliptic problems, even for unbounded right-hand sides,
including the Dirichlet problem with strictly spacelike data. Whether these
techniques, which exploit the Lorentzian geometry of the graph, extend to the
whole regime (H2), and to $f\not\equiv0$ in Theorem \ref{thm:bsc}, seems to us
a very challenging question.
One would like to know whether the minimizer is $C^{1,\alpha}$
away from the saturation region, and whether $\partial\Sigma_{u}$ is regular; the
free boundary theory developed for the elastic--plastic torsion problem
\cite{CaffarelliFriedman} is the natural model.

\smallskip

\emph{Behavior at the threshold.} In the regime (H3), what happens at the
critical value $\mu=\Lam_{f}$ of Corollary \ref{cor:cheeger}(ii)? Theorem
\ref{thm:radial}(iii) and Proposition \ref{prop:1d} show that in the radial and
in the one dimensional case the equation is still solvable at the threshold,
with a minimizer touching the constraint on a null set, and that the necessary
condition \eqref{eq:giusti} is there also sufficient. Whether this persists in a
general domain we do not know; the question is the exact analogue of
\cite[Open Problem 4.4]{P2010}, and a positive answer would turn Theorem
\ref{thm:nonexistence} into a characterization.

\smallskip

\emph{The borderline datum $L=\sg$.} Corollary \ref{cor:1dLam} settles it in
dimension one: no datum other than $f\equiv0$ is admissible. In dimension
$N\ge2$ the geometry of the set where an extension of $g$ saturates the
constraint enters the picture, see Proposition \ref{prop:Aholds}(iii) and Remark
\ref{rem:Afails}, and even the existence of a minimizer may fail. In the same
borderline case, and in the regime (H2), we do not know whether the minimizer
can saturate the constraint on a set of positive measure; by Corollary
\ref{cor:H2nosat} this is excluded as soon as $L<\sg$, and it is precisely the
regime in which the maximal surfaces of \cite{BS1982} are allowed to contain
light rays. A satisfactory treatment of this case would presumably require the
relaxed setting in $BV(\Om)$, which we have not developed here.

\section*{Acknowledgements}
 Juan C. Ortiz Chata is partially supported by FAPESP 2021/08272-6 and 2022-06050 and grant 731/2026, Paraíba State Research Foundation (FAPESQ),  Brazil. {F.\, Petitta} is partially supported by the Gruppo Nazionale per l’Analisi Matematica, la Probabilit\`a e le loro Applicazioni (GNAMPA) of the Istituto Nazionale di Alta Matematica (INdAM).   {J.\,D.\,Rossi} is financially supported by CONICET grants PIP GI No 11220150100036CO, UBACyT grant 20020160100155BA, Argentina, and by MICIU/AEI/10.13039/501100011033, grant CEX2023-001347-S, Spain.

\section*{AI declaration}
 The authors acknowledge the use of Claude Opus 5 as an AI tool for detecting
typos and mistakes and for bibliographic verification. All mathematical
arguments and proofs in the final manuscript were checked and written by
the authors, who retain full responsibility for the whole content.

\end{document}